\documentclass[12pt,reqno]{amsart}
\usepackage{enumerate,amssymb,amsfonts,latexsym,psfrag}
\usepackage{amscd,amsmath,accents}
\usepackage[height=9in,width=6.5in,centering]{geometry}

\usepackage{newtxtext}  
\usepackage{newtxmath}  
\usepackage{ogonek} 
\usepackage{bm}  

\usepackage[dvips]{graphicx}
\usepackage[usenames,dvipsnames]{color}
\ifcsname{pdfsuppresswarningpagegroup}\endcsname%
\definecolor{britishracinggreen}{rgb}{0.0, 0.26, 0.15}
\definecolor{burgundy}{rgb}{0.5, 0.0, 0.13}
\definecolor{darkblue}{rgb}{0.0, 0.0, 0.55}

\usepackage{aliascnt} 
\usepackage[implicit=true,
 colorlinks=true,linkcolor=darkblue,
 citecolor=britishracinggreen,urlcolor=burgundy]{hyperref}

\usepackage[font={small}]{caption} 

\usepackage{etoolbox}
\patchcmd{\section}{\scshape}{\scshape\large}{}{}
\patchcmd{\subsection}{\bfseries}{\scshape\large}{}{}
\patchcmd{\subsection}{-.5em}{.5em}{}{}

\usepackage{xifthen}
\newtest{\IsThereSpaceOnPage}[1]{\lengthtest{\dimexpr\textheight-\pagetotal<#1}}
\let\stdsection\section
\renewcommand\section{%
\ifthenelse{\IsThereSpaceOnPage{.25\textheight}}{\clearpage}{}\stdsection}
\let\stdsubsection\subsection
\newcommand{\SubSecSkip}{\vspace{.5\baselineskip plus 8pt minus 3pt}}
\renewcommand\subsection{\ifthenelse{\equal{\arabic{subsection}}{0}}%
  {\vspace{0pt plus 10pt}}{\SubSecSkip}\stdsubsection}

\usepackage{enumitem}
\setlist{leftmargin=2\parindent,     
         itemsep=.6ex plus 2pt minus 2pt}
\setlist[description,1]{style=sameline}
\setlist[enumerate,1]{label=\slshape{(\roman*)}}
\usepackage{braket}  
\newcommand*{\st}{ \;|\; }

\newcommand{\defn}[1]{{\everymath{\bm}\sffamily\bfseries #1}}
\newcommand{\from}{\colon}
\newcommand{\DS}{\displaystyle}
\newcommand{\C}{{\mathbb C}}
\newcommand{\D}{{\mathbb D}}
\newcommand{\N}{{\mathbb N}}
\newcommand{\R}{{\mathbb R}}
\newcommand{\Z}{{\mathbb Z}}
\newcommand{\DD}{{\mathcal D}}
\newcommand{\EE}{{\mathcal E}}
\newcommand{\OO}{{\mathcal O}}

\newcommand{\RR}{{\mathcal R}}
\renewcommand{\SS}{{\mathcal S}}
\newcommand{\SP}{\mathcal{S}{\mkern-.9mu}\mathcal{P}}

\newcommand*{\closure}[1]{%
  \mkern 1.5mu\overline{\mkern-1.5mu#1\mkern-1.5mu}\mkern 1.5mu}

\renewcommand{\Re}{\operatorname{Re}}
\renewcommand{\Im}{\operatorname{Im}}

\newcommand{\uflt}{\mathfrak{u}} 

\newcommand{\vrh}[1]{\vrule height #1 width 0pt depth 0pt}
\newcommand{\shortP}{\smash{P}\vrh{1.3ex}}
\usepackage{xifthen}
\newcommand{\Gen}[2][]{#2^{(#1)}} 
\newcommand{\Puz}[1][]{  
  \ifthenelse{\isempty{#1}}{{\mathcal P}}{\Gen[#1]{\mathcal P}}}
\newcommand{\W}[1][]{  
  \ifthenelse{\isempty{#1}}{{W}}{\Gen[#1]{W}}}
\newcommand{\Wtilde}[1][]{  
  \ifthenelse{\isempty{#1}}{\widetilde{W}}{{\widetilde{W}_{#1}}}}
\newcommand{\Ptilde}[1][]{  
  \ifthenelse{\isempty{#1}}{\widetilde{\shortP}}{{\widetilde{\shortP}_{#1}}}}
\newcommand{\Piece}[2][]{  
  \ifthenelse{\isempty{#1}}%
    {P_{#2}}
    {\Gen[#1]{P_{#2}}}
}
\renewcommand{\H}[1][]{  
  \ifthenelse{\isempty{#1}}{{H}}{\Gen[#1]{H}}}
\newcommand{\FLT}[2][]{ 
  \ifthenelse{\isempty{#1}}{\check{#2}}{\Gen[#1]{\check{#2}}}}
\newcommand{\Vslit}{V_2^*}

\newcommand{\Disk}[2][]{  
  \ifthenelse{\isempty{#1}}{\D(#2)}{\D_{#1}(#2)}}
\newcommand{\CBOX}[2]{\big\lfloor{#1},\,{#2}\,\big\rceil}

\newcommand{\out}{\mathrm{out}}
\newcommand{\tin}{\mathrm{in}}
\newcommand{\Pin}[1][]{\ifthenelse{\isempty{#1}}{\FLT{P}_\tin}{\FLT[#1]{P}_\tin}}
\newcommand{\Qin}[1][]{\ifthenelse{\isempty{#1}}{\FLT{Q}_\tin}{\FLT[#1]{Q}_\tin}}
\newcommand{\Win}[1][]{\ifthenelse{\isempty{#1}}{\FLT{W}_\tin}{\FLT[#1]{W}_\tin}}
\newcommand{\Hin}[1][]{\ifthenelse{\isempty{#1}}{\FLT{H}_\tin}{\FLT[#1]{H}_\tin}}

\newcommand{\Pout}[1][]{\ifthenelse{\isempty{#1}}{\FLT{P}_\out}{\FLT[#1]{P}_\out}}
\newcommand{\Qout}[1][]{\ifthenelse{\isempty{#1}}{\FLT{Q}_\out}{\FLT[#1]{Q}_\out}}
\newcommand{\Wout}[1][]{\ifthenelse{\isempty{#1}}{\FLT{W}_\out}{\FLT[#1]{W}_\out}}
\newcommand{\Hout}[1][]{\ifthenelse{\isempty{#1}}{\FLT{H}_\out}{\FLT[#1]{H}_\out}}
\newcommand{\Vout}[1][]{\ifthenelse{\isempty{#1}}{\FLT{V}_\out}{\FLT[#1]{V}_\out}}

\newcommand{\apPtilde}{\FLT{\smash{\widetilde{\smash{P}\vrh{1.2ex}}}\vrh{1.8ex}}}
\newcommand{\apWtilde}{\FLT{\smash{\widetilde{\smash{W}\vrh{1.3ex}}}\vrh{1.9ex}}}

\newcommand{\apF}{\FLT{F}}
\newcommand{\apz}{\FLT{z}}
\newcommand{\apw}{\FLT{w}}

\newcommand{\Vtwo}{\FLT{V}_2}

\newcommand\I{\mathrm{I}}
\newcommand\II{\mathrm{II}}
\newcommand\III{\mathrm{III}}
\newcommand\IV{\mathrm{IV}}

\renewcommand\emptyset{\varnothing}  

\newcommand\mc{\mathbb{C}}

\newcommand\mr{\mathbb{R}}
\newcommand\mh{\mathbb{H}}

\renewcommand{\setminus}{\smallsetminus}  

\renewcommand\le{\leqslant}  
\renewcommand\leq{\leqslant}
\renewcommand\ge{\geqslant}
\renewcommand\geq{\geqslant}

\DeclareMathOperator\sign{sign}
\DeclareMathOperator\dist{dist}
\DeclareMathOperator\diam{diam}
\DeclareMathOperator{\ICL}{ICU}
\newcommand{\re}{\Re}
\newcommand{\im}{\Im}
\renewcommand\i{\mathrm{i}}
\DeclareMathOperator\inter{Int\,}

\newcommand{\ie}{\text{i.e.\;\,}}

\newcommand{\eg}{\emph{e.g.}\ }

\newcommand{\dd}{\operatorname{d}}
\newcommand{\hd}{\mathrm{dim_H}}
\DeclareMathOperator{\ar}{area}

\numberwithin{equation}{section}
\numberwithin{figure}{section}

\makeatletter
\def\newaliasedtheorem#1[#2]#3{%
  \newaliascnt{#1@alt}{#2}
  \newtheorem{#1}[#1@alt]{#3}
  \expandafter\newcommand\csname #1@altname\endcsname{#3}
}
\makeatother

\newaliasedtheorem{thm}[equation]{Theorem}
\newtheorem*{mainthm}{Main Theorem}
\newaliasedtheorem{cor}[equation]{Corollary}
\newaliasedtheorem{lem}[equation]{Lemma}
\newaliasedtheorem{prop}[equation]{Proposition}
\newaliasedtheorem{Def}[equation]{Definition}

\newtheoremstyle{AlgorithmStyle}
    {}{}
    {\normalfont}                
    {}                           
    {\bfseries}               
    {.}                          
    {.5em}                       
    {}  
\theoremstyle{AlgorithmStyle}
\newaliasedtheorem{algorithm}[equation]{Algorithm}

\newtheoremstyle{RemarkStyle}
    {1.5ex plus 4pt minus 2pt}   
    {1.5ex plus 4pt minus 2pt}   
    {\normalfont}                
    {}                           
    {\itshape}                  
    {.}                          
    {.5em}                       
    {}  
\theoremstyle{RemarkStyle}
\newaliasedtheorem{rem}[equation]{Remark}

\makeatletter
\let\oldfigure\figure
\def\figure{\@ifnextchar[\figure@i \figure@ii}
\def\figure@i[#1]{\addtocounter{equation}{1}\oldfigure[#1]}  
\def\figure@ii{\addtocounter{equation}{1}\oldfigure} 
\let\oldtable\table
\def\table{\@ifnextchar[\table@i \table@ii}   
\def\table@i[#1]{\addtocounter{equation}{1}\oldtable[#1]} 
\def\table@ii{\addtocounter{equation}{1}\oldtable} 
\makeatother

\def\IMSmarkvadjust{0 pt}
\def\IMSmarkhadjust{0 pt}
\def\IMSmarkhpadding{0 pt}
\def\IMSpubltext{Published in modified form:}
\def\SBIMSMark#1#2#3{
 \font\SBF=cmss10 at 10 true pt
 \font\SBI=cmssi10 at 10 true pt
 \setbox0=\hbox{\SBF \hbox to \IMSmarkhpadding{\relax}
                Stony Brook IMS Preprint \##1}
 \setbox2=\hbox to \wd0{\hfil \SBI #2}
 \setbox4=\hbox to \wd0{\hfil \SBI #3}
 \setbox6=\hbox to \wd0{\hss
             \vbox{\hsize=\wd0 \parskip=0pt \baselineskip=10 true pt
                   \copy0 \break%
                   \copy2 \break%
                   \copy4 \break}}
 \dimen0=\ht6   \advance\dimen0 by \vsize \advance\dimen0 by 8 true pt
                \advance\dimen0 by -\pagetotal
	        \advance\dimen0 by \IMSmarkvadjust
 \dimen2=\hsize \advance\dimen2 by .25 true in
	        \advance\dimen2 by \IMSmarkhadjust

  \openin2=publishd.tex
  \ifeof2\setbox0=\hbox to 0pt{}
  \else 
     \setbox0=\hbox to 3.1 true in{
                \vbox to \ht6{\hsize=3 true in \parskip=0pt  \noindent  
                {\SBI \IMSpubltext}\hfil\break
                \textsl{Inventiones Mathematicae} (to appear 2020) 
                \vfill}}
  \fi
  \closein2
  \ht0=0pt \dp0=0pt
 \ht6=0pt \dp6=0pt
 \setbox8=\vbox to \dimen0{\vfill \hbox to \dimen2{\copy0 \hss \copy6}}
 \ht8=0pt \dp8=0pt \wd8=0pt
 \copy8
 \message{*** Stony Brook IMS Preprint #1, #2. #3 ***}
}

\begin{document}
\SBIMSMark{2017/03}{December 2017}{}

\title{On the Lebesgue measure of the Feigenbaum Julia set}
\thanks{Artem Dudko acknowledges the support by the National Science Centre,
  Poland, grant 2016/23/P/ST1/04088 under POLONEZ programme which has
  received funding from the
  EU\;\protect\includegraphics[height=.7\baselineskip]{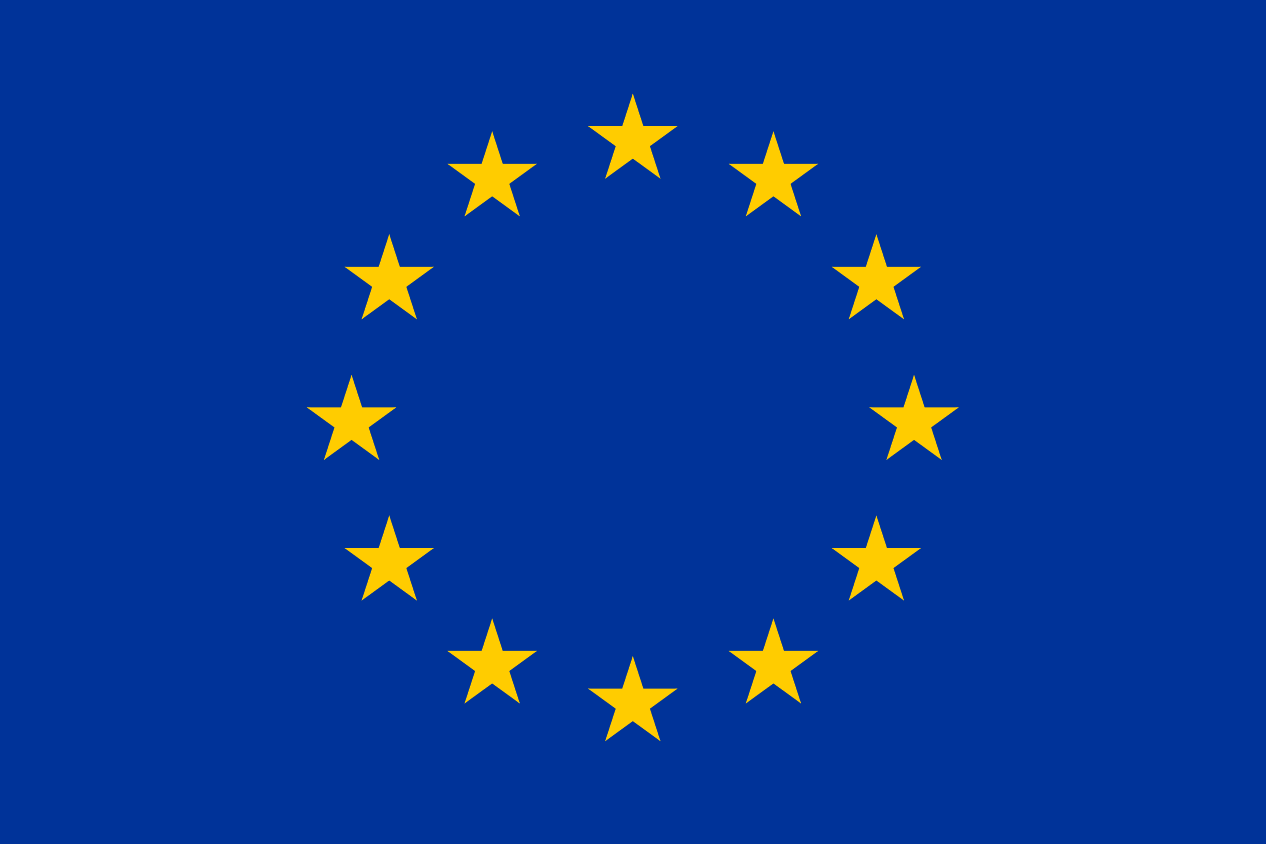}
  Horizon 2020 research and innovation programme under the MSCA grant
  agreement No. 665778}  
\author{Artem Dudko}
\address{Artem Dudko\\
Institute of Mathematics\\Polish Academy of Sciences (IMPAN)\\Warsaw, Poland}
\email{adudko@impan.pl}
\author{Scott Sutherland}
\address{Scott Sutherland\\
Institute for Mathematical Sciences\\
Stony Brook University\\
Stony Brook, New York 11794\\USA}
\email{scott@math.stonybrook.edu}

\date{}

\begin{abstract} We show that the Julia set of the Feigenbaum polynomial 
  has Hausdorff dimension less than~2 (and consequently it has zero Lebesgue
  measure).
  This solves a long-standing open question.
\end{abstract}

\maketitle
\thispagestyle{empty}

\setcounter{tocdepth}{1} 
\tableofcontents

\section{Introduction}

In \cite{AvilaLyubich-08} and \cite{AvilaLyubich-15}, Avila and Lyubich
developed a new method for studying Lebesgue area and Hausdorff dimension of
Julia sets of Feigenbaum maps. They constructed examples of Feigenbaum maps
with Julia sets of Hausdorff dimension less than two and examples of
Feigenbaum maps with Julia sets of positive area. Their approach can be used
to determine whether for a given periodic point of renormalization the Julia
set has positive area, zero area and Hausdorff dimension two, or Hausdorff
dimension less than two.
However, even for the most studied examples of
Feigenbaum maps (the Feigenbaum polynomial $f_{\mathrm{Feig}}$ and the period
doubling renormalization fixed point $F$), the calculations involved in
verification are extremely computationally complex.

In this paper we present
a new sufficient condition for the Julia set $J_F$ of $F$ to have Hausdorff
dimension $\hd(J_F)$ less than two.
Using computer-assisted means with explicit bounds on errors, we show that this
condition is satisfied. Thus, we solve a long-standing open question.

\newcommand{\refMainThm}{\hyperref[ThmMain]{main theorem}}
\begin{mainthm}\label{ThmMain}
The Hausdorff dimension of the Julia set $J_F$ of the Feigenbaum map $F$ is
less than two. In particular, the Lebesgue measure of $J_F$ is equal to
zero.
\end{mainthm}

The structure of the paper is as follows. In \autoref{SecPrelim}, we
briefly recall the definition of quadratic-like renormalization, describe the
relevant results of Avila and Lyubich, and introduce the Feigenbaum
map~$F$. We also define the sets $\tilde{X}_n$ of points whose orbits
intersect certain small neighborhoods of the origin, and denote by
$\tilde\eta_n$ the relative measures of $\tilde{X}_n$.
Using the Avila-Lyubich results,  we conclude that in order to prove
$\hd(J_F)<2$ it is sufficient to show that $\tilde\eta_n$ converges to
$0$ exponentially fast in $n$.

In \autoref{SecFeigStruct}, we describe the structure of the map
$F$.
\autoref{SecBounds} gives us distortion bounds for
certain branches of inverse iterates of $F$.
In \autoref{SecRecEst}, we state and prove the main result of the
present paper (\autoref{ThRecEst}), which gives recursive
inequalities for $\tilde{\eta}_n$.
As a result of \autoref{ThRecEst}, we obtain a sufficient condition to
show $\hd(J_F)<2$ (\autoref{PropRecMeas}).
This condition is one that can be checked by rigorous computer
estimates, which we discuss in \autoref{SecComputation}.

\medskip
The authors would like to acknowledge the invaluable assistance of Misha
Lyubich, who suggested the problem (as well as the collaboration) and
participated in many fruitful discussions,
providing constant encouragement and attention.
This paper was also significantly improved by discussions with Michael
Yampolsky, Sebastian van~Strien and Davoud Cheraghi, to whom we are quite
appreciative.

\section{Preliminaries}\label{SecPrelim}
Recall that a \defn{quadratic-like map} is a ramified covering $f\from U\to V$ of
degree~$2$, where $U\Subset V$ are topological disks in $\C$. We refer the
reader to \cite{DH}, \cite{L}, or \cite{McM} for a more detailed treatment.
For a quadratic-like  map $f$, its filled Julia set $K_f$ and Julia set
$J_f$ are given by
\begin{equation*}\label{EqJuliaDef}
  K_f=\Set{z\in U \st f^n(z)\in U \text{ for all } n\in\N}, \quad
  J_f=\partial K_f.
\end{equation*}

Let $f\from U\to V$ be quadratic-like.
The map $f$ is \defn{renormalizable of period~$n$}
if there there is an $n>1$ and $U'\subset U$ for which
$f^n\from U'\to V'=f^n(U')$ is a quadratic-like map with connected Julia set~$J'$,
and such that the sets $f^i(J')$ are either disjoint from $J'$ or
intersect it only at the $\beta$-fixed point.
In this case, $f^n|_{U'}$ is called a \defn{pre-renormalization of $f$};
the map  $\RR_n f:=\Lambda\circ f^n|_{U'}\circ\Lambda^{-1}$,
where $\Lambda$ is an appropriate rescaling of $U'$, is the
\defn{renormalization of $f$}.

An infinitely renormalizable quadratic-like map $f$ is called a
\defn{Feigenbaum map} if it has bounded combinatorics (that is, there is a
uniform bound on the periods $n$ of renormalization) and of bounded type
(the moduli of $V\setminus U$ are uniformly bounded).

\smallskip
The present paper is concerned with the Hausdorff dimension of the Julia set
of the quadratic Feigenbaum polynomial
$f_{\mathrm{Feig}}(z)=z^2+c_{\mathrm{Feig}}$, where
$c_{\mathrm{Feig}} \approx -1.4011551890$ is the limit of the sequence of real
period doubling parameters.  Discovery of
universality properties in the period doubling case during the 1970s by
Coulet \& Tresser and Feigenbaum (\cite{CT}, \cite{Feigenbaum1,
Feigenbaum2}) gave rise to the development of renormalization theory in
dynamics.  This development is well documented; for a brief overview, see
\cite[\S1.5]{Lyubich-Hairiness} and the references therein, for example.

The Julia set of $f_{\mathrm{Feig}}$ has been shown to be locally
connected (see \cite{HuJiang, Jiang}, also \cite{Buff}), although the Julia
set is ``hairy'' in the sense that it converges to the entire plane when
magnified about the critical point (see \cite[Thm.~8.7]{McM}).
In contrast to the quadratic case, the Hausdorff dimension of the Julia set of a
period-doubling Feigenbaum map tends to~2 as the order of the critical point
tends to infinity\cite{LevinSwiatek1}, while the corresponding Lebesgue measure tends to
zero \cite{LevinSwiatek2}.

\SubSecSkip
\subsection{The Avila-Lyubich trichotomy}
Let $f$ be a Feigenbaum map. Let $f_n$ denote the $n$-th pre-renormalization
of $f$, let $J_n$ be its Julia set, and let $\OO(f)$ be the
critical orbit.

\goodbreak
\setlength{\columnsep}{-7em}

Avila and Lyubich showed the existence of domains $U^n\subset V^n$ (called
``nice domains'') for which

\begin{itemize}[itemsep=.2ex plus 2pt minus 2pt]
\item $f_n(U^n)=V^n$;
\item $U^n \;\supset\; J_n\cap \OO(f)$;
\item $V^{n+1}\subset U^n$;
\item $f^k(\partial V^n)\cap V^n=\emptyset$ for all $n,k$;
\item $A^n=V^n\setminus U^n$ is ``far'' from $\OO(f)$;
\item $\ar(A^n)\asymp \ar(U^n)\asymp         \diam(U^n)^2\asymp\diam(V^n)^2$.
\end{itemize}
\noindent The construction of $U^n$ and $V^n$ involves cutting neighborhoods
of zero by equipotentials and external rays of pre-renormalizations $f_n$ of
$f$ and taking preimages under long iterates of $f_n$.

\smallskip
For each $n\in\N$, let $X_n$ be the set of points in $U^0$ that land
in $V^n$ under some iterate of $f$, and
let~$Y_{n}$ be the set of points in $A^n$ that never return to $V^n$
under iterates of $f$.
Introduce the quantities
$$\eta_{n}=\frac{\ar(X_{n})}{\ar(U^0)},
  \quad \xi_{n}=\frac{\ar(Y_{n})}{\ar(A^n)}.$$

\begin{thm}[Avila-Lyubich \cite{AvilaLyubich-08}]\label{ThAL}
 Let $f$ be a periodic point of renormalization, \ie there is a $p$ so that
 $\RR^p f=f$. Then exactly one of the following is true:
\begin{description}[leftmargin=8.5em,labelindent=\parindent,
                    itemsep=.2ex plus 1pt minus 2pt]
\item[Lean case]
   $\eta_n$ converges to $0$ exponentially fast, $\inf\xi_n>0$, and
  $\hd(J_f)<2$;
\item[Balanced case]
   $\eta_n\asymp \xi_n\asymp \frac{1}{n}$ and $\hd(J_f)=2$ with $\ar(J_f)=0$;
\item[Black Hole case]
   $\inf\eta_n>0$, $\xi_n$ converges to $0$ exponentially fast, and
  $\ar(J_f)>0$.
\end{description}
Specific bounds determining the behavior of  $\eta_n$ and $\xi_n$ depend on
the geometry of $A^n$ and $\OO(f)$.
\end{thm}
The proof of \autoref{ThAL} relies on recursive estimates involving $\eta_n$,
$\xi_n$, and the Poincar\'e series for $f_n$. In the \emph{Lean case}
(relevant for this paper), Avila and Lyubich showed existence of a constant
$C>0$ which only depends on geometric bounds for $U^n\subset V^n$ and
$\OO(f)$, such that if there exists $m$ divisible by $p$ with
$\eta_m<\xi_m/C$, then $\eta_n\to 0$ exponentially fast.

Thus, to show that for the period doubling renormalization fixed point $F$ one
has $\hd(J_F)<2$, it would be sufficient to compute this constant $C$ and find
large enough $m$ so that $\eta_m<\xi_m/C$. However, this task turns out
to be extremely computationally complex for several reasons, including:
\begin{itemize}
\item constructing the sets $U^n$ and $V^n$ is very technical, and it is
  difficult to obtain rigorous approximations of these sets computationally;
\item the geometry of $U^n$ and $V^n$ is complicated and $U^n$ is not
  compactly contained in $V^n$, making the corresponding geometric bounds
  very rough;
\item the constant $C$ is given implicitly; estimates show it can be very
  large (on the order of $10^{10}$).
\end{itemize}

In our new sufficient condition for showing $\hd(J_F)<2$, we overcome these
difficulties by using the tiling of the plane by preimages of the upper
and the lower half-planes as introduced in \cite{Buff}. In particular, we
replace the nice domains of Avila/Lyubich by the Buff tiles containing zero on
the boundary.  These tiles can be approximated quite efficiently and have good
geometric bounds. Moreover, the scale-invariant structure of the
tiling allows us to construct explicit recursive estimates for quantities
which are an analogue to $\eta_n$ directly, without using the Poincar\'e series.
While our approach allows showing $\ar(J_F)=0$ without appealing to
\autoref{ThAL},  the results of \cite{AvilaLyubich-08} give the stronger
result that $\hd(J_F)<2$. 

\subsection{The fixed point of period-doubling renormalization}\label{SubsecFixP}
Recall (see \cite{Epstein-Notes}) that the fixed point $F$ of period-doubling
renormalization is a solution of  Cvitanovi\'c-Feigenbaum equation:
\begin{equation}\label{EqCvitFeig}
\left\{
\begin{array}{lll}
F(z) & = & -\tfrac{1}{\lambda}F^2(\lambda z),\\
F(0) & = & 1,\\
F(z) & = & H(z^2),\;\text{with}\;H^{-1}(z)\;\text{univalent in}
\;\C\setminus((-\infty,-\tfrac{1}{\lambda}]\cup[\tfrac{1}{\lambda^2},\infty)),
\end{array}\right.
\end{equation}
where  $\tfrac{1}{\lambda}=2.5029\ldots$ is one of the Feigenbaum constants.

From (\ref{EqCvitFeig}) we immediately obtain
 \begin{equation}\label{EqF2m}
 F^{2^m}(z)=(-\lambda)^mF(\tfrac{z}{\lambda^m})
 \end{equation}
whenever both sides of the equation are defined.

Results of H.~Epstein \cite{Epstein-89,Epstein-Notes} imply that there
exists a domain~$\W$ containing $0$ such that $F|_{\W}$ is a
quadratic-like map
\[ F\from \W \to
  \C\setminus((-\infty,-\tfrac{1}{\lambda}]\cup[\tfrac{1}{\lambda^2},\infty)).
\]
The  \defn{$n$-th pre-renormalization $F_n$} of $F$ is the restriction
of $F^{2^n}$ onto $\W[n]=\lambda^n \W$.

\bigskip
For each $n\in\N$, let $\tilde{X}_n$ denote the set of points $z\in \W[1]$
such that $F^k(z)\in \W[n]$ for some $k\ge 0$.
Set $$\tilde\eta_n=\frac{\ar(\tilde{X}_n)}{\ar(\W[1])}.$$
Thus, $\tilde\eta_n$ is the probability that the orbit of a point randomly
chosen from $\W[1]$ with respect to Lebesgue measure will intersect~$\W[n]$.
By construction, $\tilde X_{n+1}\subset \tilde X_n$ for any $n$. Therefore, $\tilde\eta_n$ is non-increasing in $n$.

\begin{lem}\label{LmHatEtaN} If $\tilde\eta_n$ converges to $0$ exponentially
  fast then $\eta_n$ also does.
\end{lem}

\begin{proof} The properties of nice domains imply that there exists $n_0$
such that $V^{n+n_0}\subset \W[n]$ for every~$n$.
Then $X_{n+n_0}\subset \lambda^{-1}\tilde{X}_{n+n_0+1}$ for every $n$, from which
the lemma follows.
\end{proof}

\section{Structure of the Feigenbaum map \texorpdfstring{$F$}{F}}
\label{SecFeigStruct}
  For the proof of the following we refer the reader to \cite{Epstein-Notes}:

\begin{prop}\label{PropX0} Let $x_0$ be the first positive preimage of~0
  under~$F$.
Then
\[
 F(\lambda x_0)=x_0, \quad
 F(1)=-\lambda, \quad
 F(\tfrac{x_0}{\lambda})=-\tfrac{1}{\lambda}
\]
and $\tfrac{x_0}{\lambda}$ is the first positive critical point of $F$.
\end{prop}

 A map $g\from U_g\to\mc$ is called an analytic extension of a map $f\from U_f\to \mc$ if
 $f$ and $g$ are equal on some open set. An extension $\hat{f}\from S\supset
 U_f\to\mc$ of $f$ is called the \defn{maximal analytic extension} if every
 analytic extension of $f$ is a restriction of $\hat{f}$. The following
 crucial observation is also due to H.~Epstein
 (cf. \cite{Epstein-89,Epstein-Notes}; see also \cite[\S7.3]{McM}):

\begin{thm} The map $F$ has a maximal analytic extension
$\widehat{F}\from \widehat{W}\to\mc,$ where $\widehat{W}\supset \mr$ is an open simply connected set which is dense in $\mathbb C$.
\end{thm}


\begin{figure}[htb]
\centering\includegraphics[width=.75\textwidth]{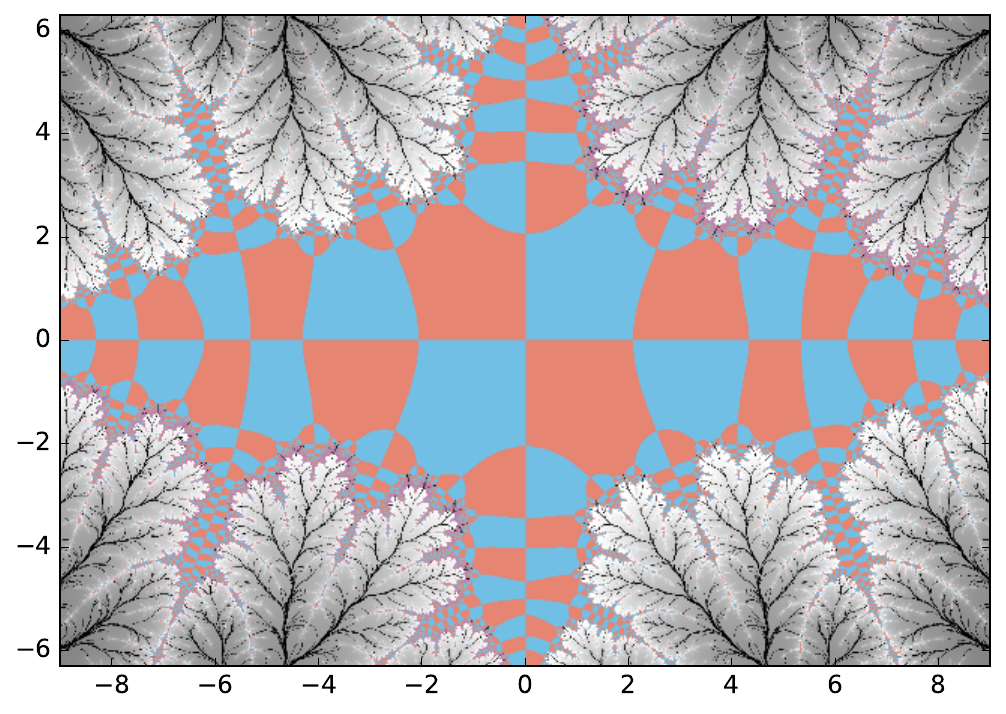}
\caption{\label{PicDomainAnalytic}
An approximation of the domain $\widehat{W}$ on which $\widehat{F}$ is defined.
The red regions are preimages of the upper half plane $\mh_+$, and the blue are
preimages of $\mh_-$.  Recall that $\widehat{W}$ is open and dense in $\C$.
Shown in shades of gray (including black and white) are points that lie in both
$\widehat{W}$ and its complement; the shading should help give some idea of
its structure.
}
\end{figure}

\noindent

Let $\mh_+=\Set{z \st \im z>0}$ be the upper half-plane, and let
$\mh_-=\Set{z \st \im z<0}$ denote the lower half-plane.

\goodbreak
 For a proof of the following, see \cite{Epstein-Notes} or \cite{Buff}:
\begin{thm}\label{ThCrit} All critical points of $\widehat{F}$ are simple. The
  critical values of $\widehat{F}$ are contained in the real axis. Moreover,
  for any $z\in \widehat{W}$ such that $\widehat{F}(z)\notin\mr$, there exists a
  bounded open set $U(z)\ni z$ such that $\widehat{F}$ is one-to-one on $U(z)$
  and $\widehat{F}(U(z))$ coincides with $\mh_+$ or $\mh_-$.
\end{thm}

\goodbreak
Following \cite{Buff}, we now introduce a combinatorial partition of $\widehat{W}$.

\begin{Def}\label{DefTiles}
Denote by $\Puz$ the set of all connected components of
$\widehat{F}^{-1}(\C\setminus \R)$.
For each non-negative integer $n$, let
  \[\Puz[n]=\Set{\lambda^n P \st P\in\Puz} \]
\end{Def}
\noindent Using the
Cvitanovi\'c-Feigenbaum equation (\ref{EqCvitFeig}), we obtain that for any non-negative integer $m$, the partition $\Puz[m]$ coincides with
the set of connected components of the preimage of $\mc\setminus\mr$ under
$\widehat{F}^{2^m}$.
\begin{Def}\label{DefTiles1} For $k\geq 0$ we will refer to connected components of $\widehat F^{-k}(\mc\setminus\mr)$ as \defn{tiles}.\\
In particular, an element of $\Puz[n]$ is a tile for any $n\ge0$,
as are the half-planes $\mh_+$ and $\mh_-$.
\end{Def}
%
\medskip

\begin{figure}[htb]
 \centering\includegraphics[width=.6\textwidth]{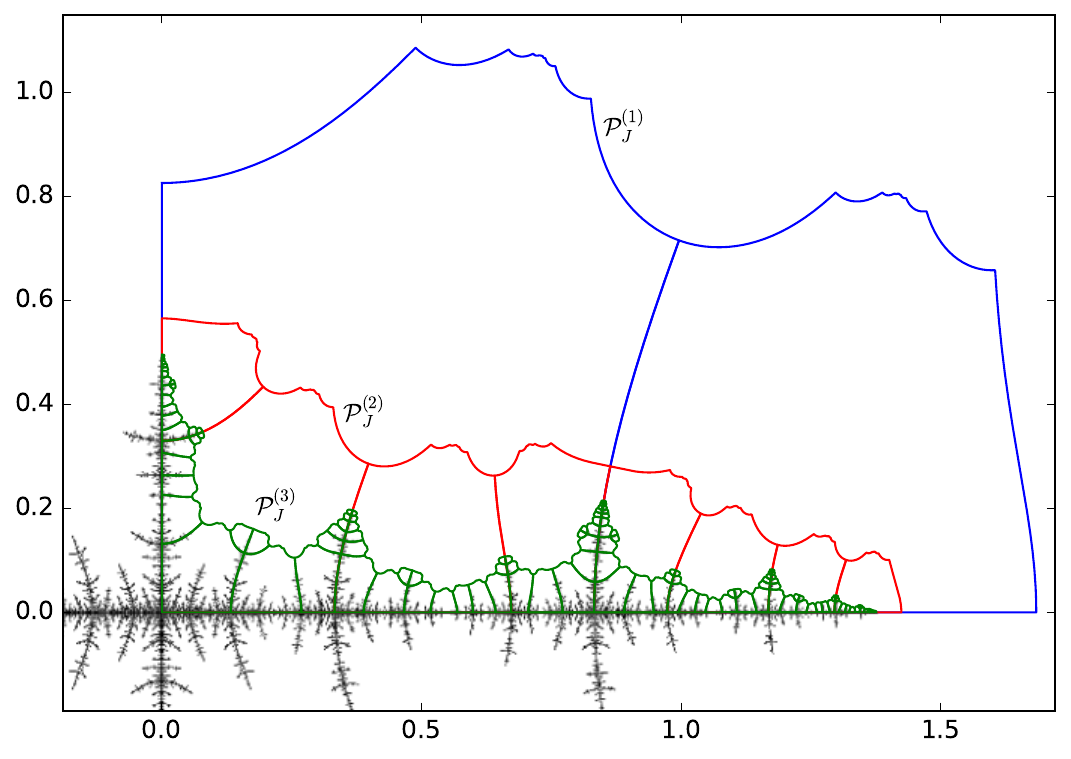}
 \caption{\label{FigTilesWithJ}
   Tiles from $\Puz[1]$~(blue), $\Puz[2]$~(red), and $\Puz[3]$~(green) in
   the first quadrant that intersect the Julia set $J_F$. Compare to
   \autoref{PicDomainAnalytic}. }
\end{figure}

Hence for any tile $P$, there is a $k\ge0$ so that the map $\widehat{F}^k$
sends $P$ bijectively onto $\mh_+$ or $\mh_-$.
Using Theorem \ref{ThCrit} we obtain the following:
\newcommand\refNestingProp{\hyperref[NestingProperty]{nesting property}}
\begin{lem}[Nesting Property]\label{NestingProperty}
Any pair of tiles are either disjoint or one is a
  subset of the other.
\end{lem}

\bigskip
Let us describe the structure of $\widehat{F}$ on the real line near the
origin. Since $\widehat{F}$ maps $[1,\tfrac{x_0}{\lambda}]$ homeomorphically
onto $[-\tfrac{1}{\lambda},-\lambda]$,
there exists a unique $a\in (1,\tfrac{x_0}{\lambda})$ such that
$\widehat{F}(a)=-\tfrac{x_0}{\lambda}.$ 

\goodbreak
A proof of the following can be found in \cite{DudkoYampolsky-16}.
\nobreak
\begin{lem}\label{LmCritPoints}
  The map $\widehat{F}$ has exactly~3 critical points the interval $(0,6)$,
  ordered as
  $0 < \tfrac{x_0}{\lambda} < \tfrac{a}{\lambda} < \tfrac{x_0}{\lambda^2}$~.
Furthermore, \[
 \widehat{F}(\tfrac{x_0}{\lambda})=-\tfrac{1}{\lambda},\quad
 \widehat{F}(\tfrac{a}{\lambda})=\tfrac{1}{\lambda^2}.\]
\end{lem}

\begin{figure}[bht]
\centering\includegraphics[height=1.75in]{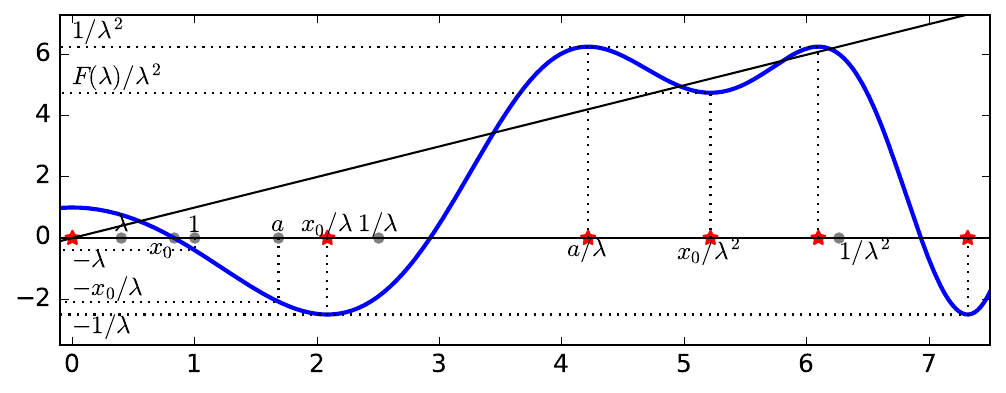} 
\caption{The graph of $\widehat{F}$, showing some dynamically important
points and their images.}
\end{figure}

Applying \autoref{ThCrit} and \autoref{LmCritPoints}, we see  that for
each of the segments 
$[0,\tfrac{x_0}{\lambda}]$, $[\tfrac{x_0}{\lambda},\tfrac{a}{\lambda}]$, and
$[\tfrac{a}{\lambda},\tfrac{x_0}{\lambda^2}]$
there is exactly one tile $P\in\Puz$ in the first quadrant which contains this
segment in its boundary. Notice that $\Puz$ has four-fold symmetry: it is
invariant under multiplication by $-1$ and under complex conjugation.

\begin{Def}\label{PieceDef}
Let $c_j$ be the non-negative real critical points of $\widehat{F}$, with
$0=c_0 < c_1 < c_2 <\ldots$, and for each $j$,
let {$\Piece{j,\I}$} denote the tile of $\Puz$ in the first quadrant with
$[c_j,c_{j+1}]$ in its boundary.

For each $K \in \Set{\II,\III,\IV}$, let {$\Piece{j,K}$} denote the tile in
quadrant $K$ symmetric to $\Piece{j,\I}$ with respect to the imaginary axis,
the origin, or the real axis, respectively.
See \autoref{fig-pieces}.
We will sometimes omit the second
index  (\eg $\Piece{2}$); in this case we will mean any of the four symmetric
tiles  $\Piece{j,\I}, \Piece{j,\II}, \Piece{j,\III},$ or $\Piece{j,\IV}$
(or the appropriate one, depending on context). \\
\indent Consistent with our earlier usage for $\Puz$,
for any set $P$ and any integer $n\ge 0$, we
let ${\Gen[n]{P}} = \lambda^n P$.
\end{Def}
\goodbreak


\begin{figure}[bht]
  \centering\includegraphics[width=0.88\textwidth]{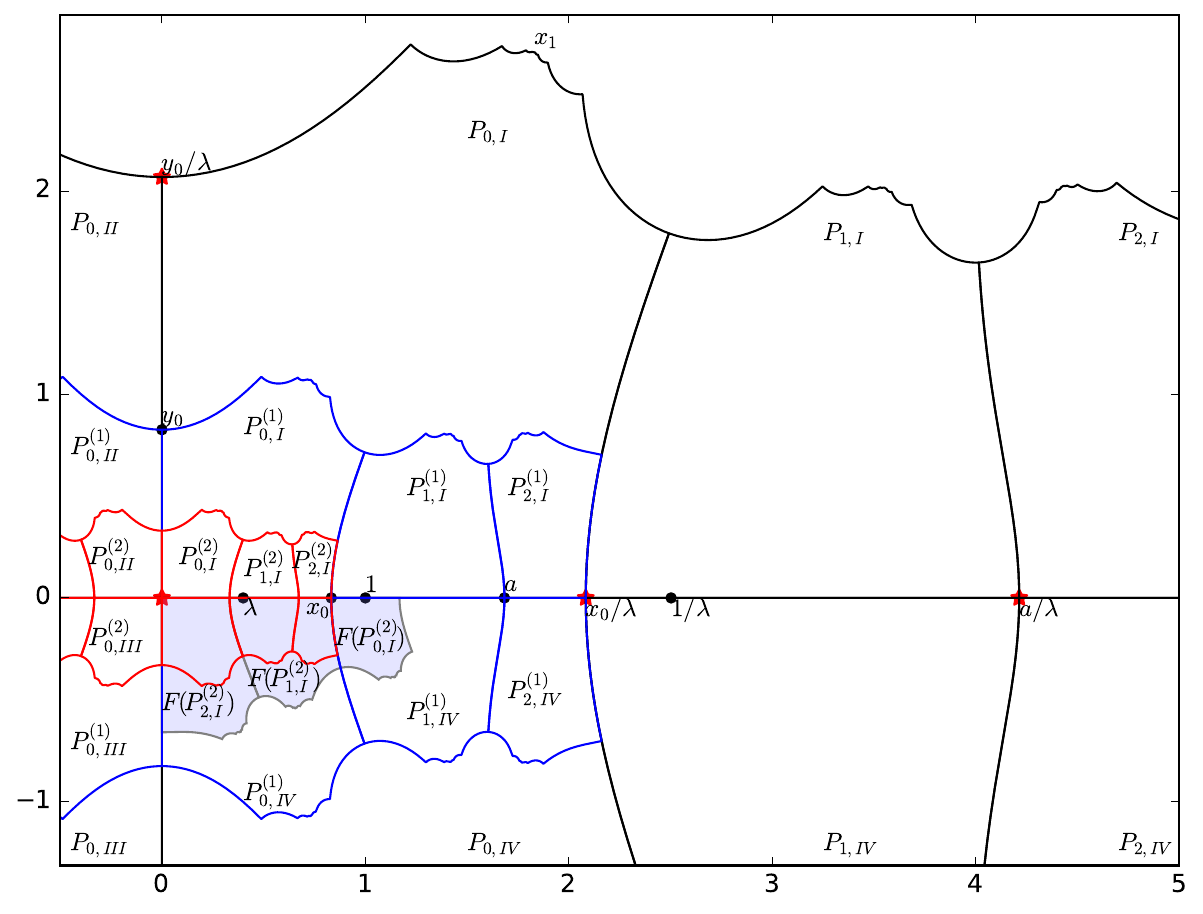} 
\caption{\label{fig-pieces}
  Illustration of \autoref{PieceDef} and
  \autoref{PropTilingStr}, showing several tiles $\Piece[n]{i,K}$ and
  some images under $\widehat{F}$.}
\end{figure}

 \begin{prop}\label{PropTilingStr} The map $\widehat{F}$ satisfies the
   following:
 \begin{enumerate}
 \item \label{FP02}
       $\widehat{F}(\Piece[2]{0,\I}) \;\subset\; \Piece[1]{1,\IV}$\;;
 \item \label{FP12}
        $\Piece[2]{1,\IV} \cup \Piece[2]{2,\IV} \;\subset\;
        \widehat{F}(\Piece[2]{1,\I}) \;\subset\;
        \Piece[1]{0,\IV} \setminus \Piece[2]{0,\IV}$\;;
 \item \label{FP13}
       $\widehat{F}(\Piece[2]{2,\I})\;\subset\; \Piece[1]{0,\IV}$\;;
 \item \label{PropTilingStr-4}
       $\widehat{F}(\Piece[1]{0,\I})=\Piece{0,\IV}$\;, \quad
       $\widehat{F}(\Piece[1]{1,\I})=\Piece{0,\III}$\;.
 \end{enumerate}
 \end{prop}

\noindent See \autoref{fig-pieces}.
We refer the reader to \cite{DudkoYampolsky-16} for the proof of
\autoref{PropTilingStr}.
Using part~\ref{PropTilingStr-4} of \autoref{PropTilingStr} we obtain

\begin{lem}\label{LmImCrit} Let $y_0$ be such that
$\closure{\Piece[1]{0,\I}} \cap \i\R=[0,y_0]$. Then
\[ \widehat{F}(y_0)=\tfrac{x_0}{\lambda},\quad
   \widehat{F}^3(\lambda y_0)=0, \quad\text{and} \quad
   \lambda y_0\in J_F.\]
\end{lem}

For the reader's convenience we list approximate values of some of
the important constants.
\begin{equation*}\label{EqConst}
 \lambda\approx 0.399535, \quad x_0\approx 0.832367 ,\quad
  \widehat{F}(\lambda)\approx 0.758923,\quad a\approx 1.683627
\end{equation*}

\bigbreak

Let $\inter S$ denote the interior of the set $S$.
The set $\W$ introduced in \autoref{SubsecFixP} can be written
\[
 \W=\inter{\closure{\Piece{0,\I}\cup\Piece{0,\II}\cup\Piece{0,\III}\cup\Piece{0,\IV}}}.
\]

\begin{Def}\label{DefWandH}
Define each of the
following quantities:
\begin{itemize}
\item For each $K \in \Set{\I,\II,\III,\IV}$, set
  $H_K=\inter \closure{\Piece{0,K}\cup \Piece{1,K}}$. 
Let $H = \inter{ \closure{H_\I \cup H_\II \cup H_\III \cup H_\IV}}$.
\item For $n\ge 0$, we write $\W[n] = \lambda^n\,\W$ and $\H[n]=\lambda^n H$.
\end{itemize}
\end{Def}

\begin{rem}\label{RemQuadrLike}
 The restriction $\widehat{F}\from \W\to
 \mc\setminus((-\infty,-\tfrac{1}{\lambda}]\cup[\tfrac{1}{\lambda^2},\infty))$
  is a quadratic-like map.
 \end{rem}

\begin{Def}\label{DefDom} Henceforth, we define $F$ to be the restriction of
  $\widehat{F}$ to $\W$. For $n\in\Z_+$, let $F_n$ denote
  the restriction of $F^{2^n}$ to $\W[n]$, that is, the
  $n$-th pre-renormalization of $F$.
\end{Def}

\begin{rem}\label{RemCritOrb}The following observations are immediate from
   the properties of $F$ described above.
\begin{enumerate}
\item If $n$ is odd then $F^n(0)\ge F^3(0)=F(\lambda)\approx 0.758923$.
\item If $F^{k}(0)\in [-\lambda^m x_0,\lambda^m x_0]$
(that is, if $F^k(0)\in \closure{\W[m+1]}$), then $k$ is divisible by $2^m$.
\item\label{RemIterW}
  For all $n\in\N$ and $1\leq k< 2^{n-1}$ one has
  $F^k(\W[n])\cap \W[n]=\emptyset.$
\end{enumerate}
\end{rem}

\begin{rem}\label{RemPullBack}
 For any finite piece of orbit
 $x=x_0,x_1=\widehat{F}(x),\ldots,x_k=\widehat{F}^k(x)$ such that $x_k$ belongs
 to the closure of
 some tile $T$ and  $D\widehat{F}^k(x_0)\neq 0$,
 one can univalently pull back $T$ along $x_0,\ldots,x_k$ in a unique way. In
 particular, this is true under the condition that $x_i\in W$ and $x_i\neq 0$
 for $1\le i \le k-1$.
\end{rem}

\section{Distortion bounds}\label{SecBounds}
\subsection{Copies of tiles}

\begin{Def}\label{DefCopy} 
We will say that a tile $Q$ is a
\defn{copy of the tile $P$ under $F^k$} if there is a non-negative
integer $k$ such that $F^k(Q)=P$.

A copy $P$ of $\Piece[m]{0}$ under $F^k$ is \defn{primitive} if
$F^j(P)\cap \W[m]=\emptyset$ for all  $0\leq j<k$.

A copy $P$ of $\Piece[m]{0}$ under $F^k$ is \defn{separated}
if there exists $0\leq j<k$ with
$F^j(P)\subset \W[m]$ and $F^j(P)\cap \Gen[m-1]{J_F}=\emptyset$ for the maximal such $j$.
\end{Def}

As we shall see, separated copies of $\Piece[m]{0}$ are called this because
they stay away from relevant parts of the postcritical set.
Separated copies play an important role for us, in that they allow us to
have control on the distortion of tiles under iteration of $f$.

\begin{rem}\label{RemSepCop} We make the following useful observations.
\begin{enumerate}
\item If $Q$ is a copy of $P$ under $F^k$ then $F^k:Q\to P$ is a bijection.
\item\label{CopiesofSepAreSep}
  If $Q$ is a copy of a separated copy $T$ then $Q$ is separated.
\item\label{IterofSep} Let $T$ be a separated copy of $\Piece[m]{0}$ with $F^k(T)=\Piece[m]{0}$.
  Then for each $j\le k$, $F^j(T)$ is either a primitive or a separated copy
  of~$\Piece[m]{0}$.
  In particular, the set~$\Piece[m]{0}$ is a primitive copy of itself.
\item Copies of primitive copies need not be primitive or separated.
  For example,  let $T =F(\Piece[2]{0,\I})$, so  $F^2(T)=\Piece[1]{0,\II}$
  and $T \subset \Piece[1]{1,\IV} \not\subset \W[1]$; see \autoref{fig-pieces}.
  Hence $T$ is a primitive copy of $\Piece[1]{0,\II}$.\\
  But $\Piece[2]{0,\I} \subset \W[1]$ and intersects $J_F$, so
  $\Piece[2]{0,\I}$ is neither a primitive nor a separated copy of
  $\Piece[1]{0,\II}$.

\end{enumerate}
\end{rem}

\begin{lem}\label{LmCopiesInWn}
Let $T$ be a copy of $\Piece[m]{0}$ under $F^k$ with $T\subset \W[n-1]$ for some
$n\le m\in\N$. If $T\cap\Gen[n-1]{J_F} \neq \emptyset$, then
$k$ is divisible by  $2^{n-1}$.
Moreover,  $F_{n-1}^\ell$
is defined on $T$ for $\ell=k/2^{n-1}$ and thus $F_{n-1}^\ell(T)=\Piece[m]{0}$.
\end{lem}

\begin{proof} Let $T$ be as in the conditions of the lemma. Assume that $k$ is not divisible by~$2^{n-1}$.  Let
$r$ be the remainder of $k$ modulo $2^{n-1}$ and let $j=2^{n-1}-r$. Then
$F^{j}(\Piece[m]{0})=F^{2^{n-1}s}(T)$ for some integer~$s$, therefore,
\[
   F^{j}(P_{0,J}^{(m)})\cap \Gen[n-1]{J_F}\neq\emptyset.
\]
By the nesting property (\autoref{NestingProperty}), we see that
$F^{j}(\Piece[m]{0})\subset \W[n-1]$ and hence $F^j(0)\in\W[n-1]$.
\autoref{RemCritOrb} implies that $j=2^{n-2}$. But
$F^{2^{n-2}}(0)=(-\lambda)^{n-2}$ does not belong to $\W[n-1]$. This
contradiction proves the first statement of the lemma.

To prove the second statement we show by induction that for all
$0\le j<\ell$ one has $F_{n-1}^j(T)\subset \W[n-1]$ and
$F_{n-1}^j(T)\cap\Gen[n-1]{J_F}\neq\emptyset$. The base
case of the induction $j=0$ is given by
the conditions of the lemma. The induction step follows from the
\refNestingProp\ and the fact that $\Gen[n-1]{J_F}$ is $F_{n-1}$-invariant.
\end{proof}

\begin{lem}\label{LmSeparated}
Let $T$ be a separated copy of $\Piece[m]{0}$.
Then $\DS T\cap \Gen[m-1]{J_F}=\emptyset.$
\end{lem}

\begin{proof} Assume that there exists a separated copy $T$ of $\Piece[m]{0}$ such that $\DS T\cap \Gen[m-1]{J_F}\neq\emptyset.$ Let $k$ be such that $F^k(T)=\Piece[m]{0}$.
Set $\ell=\lfloor \tfrac{k-1}{2^{m-1}} \rfloor$, where
$\lfloor\cdot\rfloor$ denotes the integer part of a number,
and set $T_r=F^{2^{m-1}r}(T)$ for $0\le r \le \ell.$ Then
$T_r\cap \Gen[m-1]{J_F}\neq \emptyset$, and thus by the \refNestingProp,
$T_r\subset \Gen[m]{H}$ for all $r\le \ell$
(recall the definition of $H$ in \autoref{DefWandH}).

\begin{figure}[htbp]
\centering\includegraphics[width=0.7\textwidth]{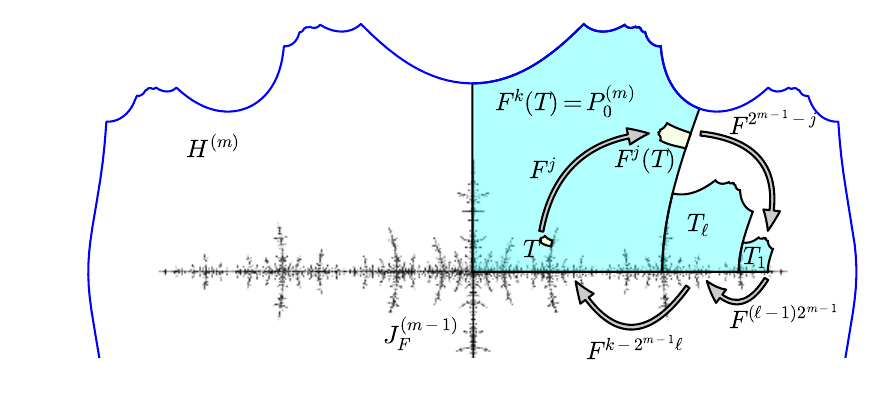}
\caption{\autoref{LmSeparated} shows the above situation does not occur, that
  is, a separated copy $T$ of $\Piece[m]{0}$ cannot intersect~$\Gen[m-1]J_F$.
  (This illustrates the case $r=0$ of the proof.)}
\end{figure}

Since $T$ is separated there exists $0\le r\le \ell$ and
$1\le j < 2^{m-1}$ such that $F^j(T_r)\subset \W[m]$ and
$F^j(T_r)\cap \Gen[m-1]{J_F}=\emptyset$.
Taking into account that
$$T_r\subset \Gen[m]{H} \subset \Gen[m-1]{W}
  \;\;\text{and}\;\; F^j(T_r)\subset\Gen[m]{W} \subset \Gen[m-1]{W},$$
from \autoref{RemCritOrb}\ref{RemIterW} we obtain that $j\geq 2^{m-2}$.
Similarly, since
\[ F^{2^{m-1}}(T_r)\subset
   F^{2^{m-1}}(H^{(m)})=(-\lambda)^{m-1}F(\Gen[1]{H})\subset \Gen[m-1]{W}
\]
we know that $2^{m-1}-j\geq 2^{m-2}$ by again using
\autoref{RemCritOrb}\ref{RemIterW}. 
It follows that $j=2^{m-2}$. 
But
\[
F^{2^{m-2}}(\Gen[m]{H})\;=\;
(-\lambda)^{m-2}F(\Gen[2]{H}) \;\subset\; \W[m-2]\setminus\W[m]
\]
(see \autoref{PropTilingStr}). 
This contradiction finishes the proof.
\end{proof}

\subsection{Koebe space}
In what follows, we use $\sign(P)$ to represent the sign of the real part of
the points in $P$, that is, $\sign(P)=-1$ if $P$ is in the left half-plane,
$+1$ otherwise.

\begin{prop}\label{PropC2}
Let $T$ be a primitive or a separated copy of $\Piece[m]{0}$ under $F^k$
with $m\ge 2$. Then the inverse branch $\phi\from \Piece[m]{0}\to T$ of
$F^k$ analytically continues to a univalent map on
$\sign(\Piece[m]{0})\,\lambda^m\C_\lambda$
where\footnote{%
  The set $\C_\lambda$ should not be confused with the
  slightly larger set
  $F(\W)=\C\setminus\left(
      (-\infty,-\tfrac{1}{\lambda}]
      \cup[\tfrac{1}{\lambda^2},\infty)\right)$
.}
\[
\C_\lambda=\C\setminus\left(
  (-\infty,-\tfrac{1}{\lambda}]\cup[\tfrac{F(\lambda)}{\lambda^2},\infty)
 \right).
\]
\end{prop}
\begin{figure}[htbp]
\centering\includegraphics[height=6\baselineskip]{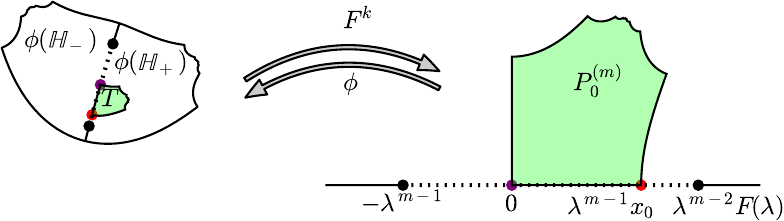}
\caption{Illustration of the statement of \autoref{PropC2},
  which says (roughly) that primitive and separated copies of $\Piece[m]{0}$
  have a definite Koebe space around them. }
\end{figure}
\begin{proof}
 Assume that the statement of \autoref{PropC2} is false.
Let $m\ge 2$ and $T$ be either a primitive or a separated copy of $\Piece[m]{0}$ such that the
inverse branch of $F^{-k}\from \Piece[m]{0}\to T$ does not have a univalent
continuation on $\sign(\Piece[m]{0})\,\lambda^m\C_\lambda$. Take
$k$ to be the minimal integer such that there is no such continuation.

Since $F$ is real analytic, without loss of generality we may also assume
that $\Piece[m]{0}\subset \mh_+$, that is, it is $\Piece[m]{0,\I}$ or
$\Piece[m]{0,\II}$. By \autoref{RemPullBack},
$F^{-k}$ has a univalent analytic continuation $\phi$ on $\mh_+$ which can be
 extended to a continuous function on $\closure{\mh_+}$.
Since $k$ is defined to be the minimal value such that $F^{-k}$ does not
extend, the unique critical point $0$ must be in $\closure{\phi(\mh_+)}$;
moreover,
 \begin{equation}\label{EqFk01}
 F^k(0)\in \sign(\Piece[m]{0})(-\lambda^{m-1},\lambda^{m-2}F(\lambda)),
 \;\;\text{ and in particular}\;\;
 |F^k(0)|<\lambda^{m-2}F(\lambda).
\end{equation}

Suppose that $\phi(0)\notin \mr$. Let $\ell$ be the minimal integer such that
$F^{\ell+1}(\phi(0))\in\R$, so $F^\ell(\phi(0))\in\i\R$ and
$\closure{F^\ell(T)}\cap\i\R$ is a non-empty segment.
Since $T$ is primitive or separated, $\closure {F^\ell(T)}$ cannot contain $0$. Therefore,
$\closure{F^\ell(T)}$ is disjoint from $\R$.
Let $[x,y]$ denote the segment $\closure{F^\ell(T)}\cap\i\mr$ where $|x|<|y|$.
See \autoref{FigEq0inI}.

\begin{figure}[htbp]
\centering\includegraphics[height=6\baselineskip]{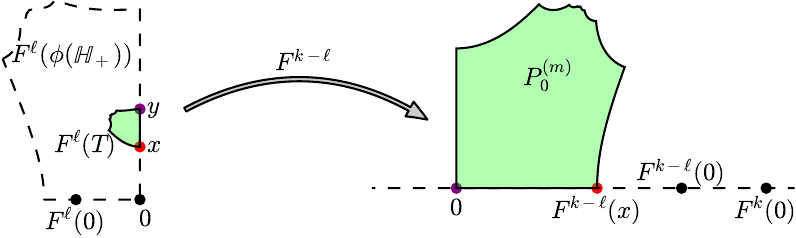}
\caption{\label{FigEq0inI} Illustration of formula~\eqref{Eq0inI}:
  if $\phi(0)\not\in\R$, then $F^{k-\ell}(0)$ lies between $0$ and $F^k(0)$. }
\end{figure}

Clearly, $F^{k-\ell}(x)$ is real and $F^{k-\ell}(x)\subset \closure{\Piece[m]{0}}$.
It follows that
$(x,0]\cup [0,F^\ell(0))\subset F^\ell(\phi(\closure{\mh_+}))$
and $F^{k-\ell}$ is one-to-one on $(x,0]\cup [0,F^\ell(0))$. Thus if $\ell$
is nonzero, we will have
\begin{equation}\label{Eq0inI}
F^{k-\ell}(0) \;\in\; F^{k-\ell}\bigl((x,0]\cup [0,F^\ell(0))\bigr)
  =(F^{k-\ell}(x),F^{k}(0)) \;\subset\; (0,F^{k}(0)),
\end{equation}
since the segment $(F^{k-\ell}(x),F^{k}(0))$ is disjoint from
$\closure{\Piece[m]{0}}$.
By the definition of $k$, we obtain that $\ell=0$.  Consequently,
$\phi(0)\in \i\mr\cup\mr$.

\medskip
By \autoref{RemCritOrb} and \eqref{EqFk01}, $k$ is divisible by $2^{m-2}$.
If $k\ge 2^{m+1}$, then since
$0\in\closure{\phi(\mh_+)}$ the \refNestingProp\ tells us that we must have
$T\subset\phi(\mh_+)\subset \W[m+1]$. In particular, $T$ is not primitive and so must be separated.
Since $\phi(0)$ must be either real or purely imaginary, and since
$$\W[m+1]\cap \Gen[m-1]{J_F} \;\supset\; \W[m+1]\cap (\i\mr\cup \mr),$$
we obtain that
$T\cap \Gen[m-1]{J_F}\neq\emptyset$, contradicting the hypothesis of $T$ being
separated.  Hence, we must have $k<2^{m+1}$.

This leaves us with four possibilities, each of which we rule out now.
\goodbreak

\begin{description}
\item[Case $\mathbf{k=2^m}$]
In this case, $T\subset\phi(\mh_+)=\Piece[m]{0,K}$ for some $K$.
In particular, $T$ is not primitive.
Since
$F^{2^m}\!\left(\sign(\Piece[m]{0,K})\lambda^m x_0\right)=0$
and $\sign(\Piece[m]{0,K})\lambda^m x_0\in \partial \phi(\mh_+)$,
we have $\sign(\Piece[m]{0,K})\lambda^m x_0=\phi(0)\in\partial T$
and therefore $T\cap J_F^{(m-1)}\neq \emptyset$. From \autoref{LmSeparated},
this means $T$ is not separated, giving a contradiction.

\begin{figure}[htbp]
 \centering\includegraphics[height=6.0\baselineskip]{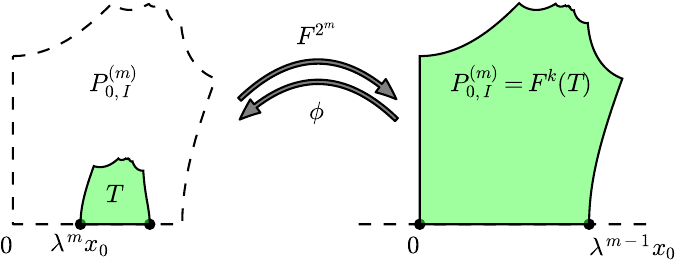}
 \caption{A case where $k=2^m$ in the proof of \autoref{PropC2}.}
\end{figure}

\item[Case $\mathbf{k=2^{m-1}}$]
 Then $\phi(\mh_+)=\Piece[m-1]{0}$. Since by \eqref{EqF2m} we
 have $F^{2^m}(z)=(-\lambda)^mF(z/\lambda^m)$, we can apply
 \autoref{PropTilingStr} to see that
 $T=\pm\lambda^{m-1} F(\Piece[2]{i,K})$ with $i=0$ or $i=1$ and $K=\I$ or
 $K=\IV$. Since $T$ is either primitive or separated, $i\neq 1$. Again using \eqref{EqF2m} we
 obtain
 $$
\Piece[m]{0}=F^{2^{m-1}}(T)=(-\lambda)^{m-1}F^2(\Piece[2]{0,K})
  =(-\lambda)^{m}F(\Piece[1]{0,K})$$
and thus $\sign(\Piece[m]{0})=(-1)^m$. This contradicts \eqref{EqFk01}, since
$F^{2^{m-1}}(0)=(-\lambda)^{m-1}$.

\begin{figure}[bht]
\centering\includegraphics[height=6.0\baselineskip]{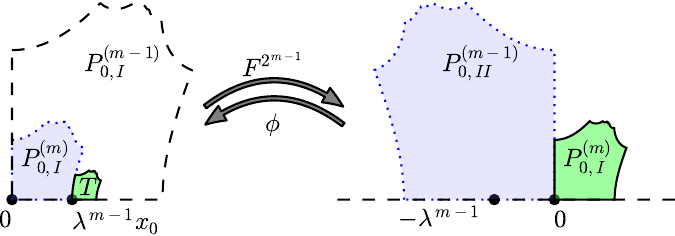}
\caption{In the proof of \autoref{PropC2}, the case where $k=2^{m-1}$.}
\end{figure}

\goodbreak
\item[Case $\mathbf{k=3\cdot 2^{m-1}}$]
Here we have $\Piece[m]{0,K}\supset\phi(\mh_+)\supset \Piece[m+1]{0,K}$ for
some $K$.
In particular, $T$ is not primitive. Moreover, by definition of $y_0$ (see \autoref{LmImCrit}),
$\closure{\phi(\mh_+)}$ will contain either $y=\lambda^my_0$ or
$y=-\lambda^my_0$ which are points of $\Gen[m-1]{J_F}$.
Since
$ F^{3\cdot 2^{m-1}}(y)=(-\lambda)^{m-1}F^3(\lambda y_0)=0$,
we obtain that $y=\phi(0)\in\closure T$, and hence
$T$ intersects $\Gen[m-1]{J_F}$, contradicting the hypothesis that $T$ is
separated.

\begin{figure}[htbp]
\centering\includegraphics[height=6.0\baselineskip]{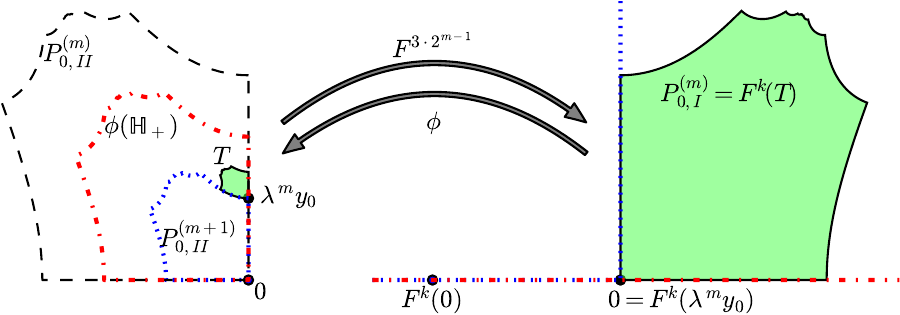}
\caption{The case $k=3\cdot 2^{m-1}$ from the proof of \autoref{PropC2}.}
\end{figure}

\item[Case $\mathbf{k=2^{m-2}j}$ where $\mathbf{j}$ is odd]
\autoref{RemCritOrb} gives $|F^k(0)|=\lambda^{m-2}|F^j(0)|\geq\lambda^{m-2}F(\lambda)$, contradicting
\eqref{EqFk01}.
\end{description}

\noindent
Since all the possibilities for $k$ lead to a contradition, we have
established the proposition.
\end{proof}
\goodbreak

\noindent Informally speaking, \autoref{PropC2} tells us that the
inverse branch of the iterate of $F$ corresponding to a separated copy $T$
admits an analytic continuation to a region with a definite Koebe space
around~$T$.

\subsection{Distortion of measurable sets}\label{SubsecBounds}
The Koebe Distortion Theorem (see \cite{Duren}, e.g.) implies that there
exists a constant $C>0$ such that for any univalent function $\phi$ on
 $\C_\lambda=\C\setminus\left(
  (-\infty,-\tfrac{1}{\lambda}]\cup[\tfrac{F(\lambda)}{\lambda^2},\infty)
 \right)$ one has:
$$\frac{|\phi'(z)|}{|\phi'(w)|}\leq C,\;\;\text{for any}\;\;z,w\in \Piece{0}.$$
Recall that $\Piece{0}$ is used to denote any of the tiles
$\Piece{0,\I}$, $\Piece{0,\II}$, $\Piece{0,\III}$, or $\Piece{0,\IV}$
(see \autoref{PieceDef}).
As a consequence of \autoref{PropC2}, one can show the following.
\begin{cor}\label{PropKoebeBA0}
 Let $A$ and $B$ be two measurable subsets of $\Piece{0}$ of positive
 measure and let $T$ be a primitive or a separated copy of
 $\Piece[m]{0}$ under $F^k$ for some $k\ge0$ and $m\ge2$.
 Then
\[
\frac{\ar(F^{-k}(\Gen[m]{B})\cap T)}
     {\ar(F^{-k}(\Gen[m]{A})\cap T)}
\;\le\; C^4\frac{\ar(B)}{\ar(A)}.
\]
\end{cor}

\noindent
However, we will need a slightly sharper version of this result, where the
constant depends more explicitly on the set $A$.
We devote the remainder of this section to establishing it.

\medskip
Consider the slit plane $\C_\lambda$.
From the Koebe Distortion Theorem, the function
$$C(z,w)= \sup \Set{ \frac{\lvert\phi'(z)\rvert}{\lvert\phi'(w)\rvert} \st
                     \phi\from\C_\lambda\to \C\;\;\text{is univalent} }$$
is nonzero and finite for all $z$ and $w$.

Fix a univalent map $\phi$ on $\C_\lambda$.
For every $w\in\C_\lambda$, there is a conformal isomorphism
$H_w\from \C_\lambda\to\D$ such that $H_w(w)=0$ and $H'_w(w)$ is a positive real number.
Thus, we can write $\phi$ as a composition
$\phi=\varphi\circ H_w$,
where $\varphi$ is a univalent map on the unit disk $\D$.

\smallskip
Applying the Koebe Distortion Theorem, we have
\[ \frac{1+|H_w(z)|}{(1-|H_w(z)|)^3} \;\ge\;
   \frac{|\varphi'(H_w(z))|}{|\varphi'(0)|} \;\ge\;
   \frac{1-|H_w(z)|}{(1+|H_w(z)|)^3}.
\]
In particular,
\[ \frac{|\phi'(z)|}{|\phi'(w)|} \;\le\; C(z,w) \;\le\;
   \frac{1+|H_w(z)|}{(1-|H_w(z)|)^3}\frac{|H_w'(z)|}{|H_w'(w)|}.
\]
The latter gives us a way to estimate $C(z,w)$ from above.

\smallskip
Fix two measurable subsets $A$ and $B$ of $\Piece{0}$ of positive measure.
For any $z\in B$ we have
\[
\ar(\phi(A))= \int\limits_{A}|\phi'(w)|^2|\dd w|^2
 \ge |\phi'(z)|^2\int\limits_{A} \frac{|\dd w|^2}{(C(z,w))^2} .
\]
For notation, set
$g_A(z)=\int\limits_{A}|\dd w|^2/C^2(z,w)$. Then we obtain
\begin{align*}
\ar(\phi(B))=\int\limits_B|\phi'(z)|^2|\dd z|^2
      \;\le\; \ar(\phi(A))\int\limits_B \frac{|\dd z|^2}{g_A(z)}
      \;\le\; M(A)\ar(\phi(A))\ar(B),
\end{align*}
where
\begin{equation*}\label{EqMA}  
  M(A)=\frac{1}{\inf\Set{g_A(z) \st  z\in \Piece{0}}}.
\end{equation*}
Observe that for two measurable sets $A\subset A'\subset \Piece{0}$, we have
$M(A)\ge M(A')$,  since for all $z$ we have $g_A(z)\le g_{A'}(z)$.

\smallskip
As a consequence of \autoref{PropC2} and the previous discussion, we obtain
the following.

\begin{cor}\label{PropKoebeBA}
 Let $A,B$ be two measurable subsets of $\Piece{0}$ of positive measure and
 let $T$ be a primitive or a separated copy of $\Piece[m]{0}$ under $F^k$
 for some $k\geq 0$ and $m\ge2$.
Then
\begin{equation*}\label{EqKoebe}
\frac{\ar( F^{-k}(\Gen[m]{B})\cap T)}
     {\ar(F^{-k}(\Gen[m]{A})\cap T)}
\;\le\; M(A)\ar(B).
\end{equation*}
 Moreover, if $A_1\subset A_2$ then $M(A_1)\ge M(A_2)$.
\end{cor}

\begin{rem}
While we could write the quantity ~$M(A)\ar(B)$~ more conventionally as
$K_A\frac{\ar(B)}{\ar(A)}$, it is more convenient to use $M(A)$ since
$A_1\subset A_2$ does not imply $K_{A_1} \ge K_{A_2}$.
\end{rem}

\section{Recursive estimates}\label{SecRecEst}

Recall that $F$ is a quadratic-like map
$F\from \W\to
  \mc\setminus((-\infty,-\tfrac{1}{\lambda}]\cup[\tfrac{1}{\lambda^2},\infty)),
$
and for a point $z\in\W$, by the \defn{forward orbit of $z$} we mean the set
$\Set{F^k(z) \st k\in\N\cup\{0\}\;\;\text{such that $F^k(z)$ is defined}}$.
Recall from \autoref{DefDom} that $F_n$ denotes the $n$-th
pre-renormalization of $F$, that is, the restriction of $F^{2^n}$ to $\W[n]$.

\medskip
\begin{Def}\label{DefXnYn}
Define the following (some of which we have referred to in
\autoref{SecPrelim}).

\begin{itemize}[itemsep=.15ex plus 2pt minus 2pt]
\item Let $\tilde{X}_n$ be the set of points in $\W[1]$ that eventually land in
  $\W[n]$:
  \[ \tilde{X}_n =\Set{ z\in\W[1] \st F^k(z)\in\W[n] \;\text{for some $k\ge0$}} .\]

\item Denote by $\tilde\eta_n$ the relative measure of $\tilde{X}_n$ in
  $\W[1]$:   \[ \tilde\eta_{n}=\DS \frac{\ar(\tilde{X}_{n})}{\ar(\W[1])}. \]

\item Let $X_{n,m}$ be the set of points in  $\W[n]$ whose forward orbits
under $F_{n-1}$ intersect $\W[n+m]$:
    \[ X_{n,m} = \lambda^{n-1} \tilde{X}_{m+1}
    =\Set{ z\in\W[n] \st F^k_{n-1}(z) \in \W[n+m]~\text{for some $k\ge0$}}. \]
\goodbreak
\item Let $Y_n$ denote the set of points in $\W[n]$ whose forward orbits
  never return to $\W[n]$:
  \[ Y_n = \Set{ z\in\W[n] \st F^{k}(z) \not\in \W[n] \;\text{for $k \ge 1$}} .\]
\item $\Sigma_n$ is the set of points in $\W[n]$ whose forward orbits under
  $F_{n-1}$ intersect
 $Y_n$:
 \[ \Sigma_{n}=\Set{z\in\W[n] \st F_{n-1}^j(z)\in Y_n
     \;\;\text{for some $j\ge0$}} \]
\item $\Sigma_{n,m}$ is the subset of $\Sigma_n$ with orbits that avoid $\W[n+m]$:
 \[ \Sigma_{n,m}=\Set{z\in\Sigma_n \st
        F_{n-1}^k(z)\notin \W[n+m] \;\text{for all $k\ge 0$} } .\]
%
\item For measurable sets $A$, let $M(A)$ be as in
  \autoref{PropKoebeBA}.  Then define
\[ M_{n,m}=M\bigl( (\lambda^{-n}\Sigma_{n,m}) \cap \Piece{0,\I} \bigr)
  \quad\text{and}\quad
   M_n =  M\bigl( (\lambda^{-n}\Sigma_{n}) \cap \Piece{0,\I} \bigr). \]
\end{itemize}
\end{Def}

%

\goodbreak
\begin{rem}\label{ObsUnm}
The following observations are immediate:
\begin{enumerate}
\item \label{XisPcopies} $\tilde{X}_{n}$ is the union of all
  primitive copies of $\Piece[n]{0}$ that lie inside $\W[1]$,
  together with a countable collection of analytic curves (which form
  parts of the boundaries of these copies).
\item\label{Etam+1} The relative measure of $X_{n,m}$ in
  $\W[n]$ is equal to $\tilde{\eta}_{m+1}$:
  \qquad$\DS \frac{\ar(X_{n,m})}{\ar(\W[n])}=\tilde\eta_{m+1} $.
\smallskip
\item  By construction, $X_{n,m}\cap \Sigma_{n,m}=\emptyset$ for all $n,m$.
\item  $\DS \Sigma_n = \bigcup_{m\in\N} \Sigma_{n,m}$.

\end{enumerate}
\end{rem}


\smallskip
\begin{thm}\label{ThRecEst}
For every $n\geq 2$ and $m\geq 1$, one has
\[
   \tilde\eta_{n+m}\le
   M_{n,m} \ar(\Piece{0,\I}) \,\tilde\eta_n\, \tilde\eta_{m+1}
\]
\end{thm}
\medskip


Before proving \autoref{ThRecEst}, let us formulate its main corollary.

\begin{cor}\label{PropRecMeas}
If for some $n\geq 2$ one has $\tilde\eta_n M_n\ar(\Piece{0,\I})<1$,
then the Hausdorff dimension of $J_F$ is less than 2.
\end{cor}

\begin{proof} Let $n$ be such that $\tilde\eta_nM_n\ar(P_{0,\I})<1$.
Then there is an $m$ for which $\tilde\eta_n M_{n,m}\ar(\Piece{0,\I})<1$; let
$\gamma =  \tilde\eta_n M_{n,m}\ar(\Piece{0,\I})$ for this value of $n$ and $m$.

By construction  $\Sigma_{n,k}\subset \Sigma_{n,\ell}$ whenever $k<\ell$, so
\autoref{PropKoebeBA} tells us that
$$M_{n,rn+m}\le M_{n,m} \quad\text{for every $r\in\N$}.$$
Using \autoref{ThRecEst}
and writing $rn+m$ as $n+((r-1)n+m)$, we deduce that
\[
  \tilde{\eta}_{rn+m+1}\le \tilde{\eta}_{rn+m}\le
   M_{n,(r-1)n+m} \ar(\Piece{0,\I}) \,\tilde\eta_n\, \tilde\eta_{(r-1)n+m+1}
  \le\gamma\,\tilde{\eta}_{(r-1)n+m+1} \quad
  \text{for every $r\in\N$}.
\]
As a result, $\tilde\eta_k$ converges to zero exponentially fast.
Consequently, \autoref{LmHatEtaN} tells us that the parameter
$\eta_k$ of the Avila-Lyubich Trichotomy (\autoref{ThAL}) also
converges to zero exponentially fast.  Thus, $F$ is in the lean case and
the Hausdorff dimension of $J_F$ is less than 2.
\end{proof}

\subsection{The proof of \autoref{ThRecEst}}
First, let us prove some auxiliary lemmas.

\begin{lem}\label{LmSeparatedness}
Let $T$ be a copy of $\Piece[m+n]{0}$ with $T\subset\W[1]\setminus X_{n,m}$.
Then there is a $k\ge0$ and a primitive or separated copy $Q$ of
$\Piece[n]{0}$ under $F^k$ so that
$T \subset Q \subset \W[1]$, with
$F^k(T)\subset X_{n,m}$ but $Q\not\subset \tilde{X}_{n+m}$.

\end{lem}
\begin{proof} Let $F^r(T)=\Piece[m+n]{0}$.
Let $j$ be the minimal number for which $F^j(T)$ intersects $X_{n,m}$
with $F^j(T)\subset\W[n]$. We now show that we must have
$F^j(T)\subset X_{n,m}$.
Since $X_{n,m}$ consists of copies of $\W[n+m]$, the \refNestingProp\
tells us that either $F^{j}(T)\subset X_{n,m}$ or
it contains a copy of $\Piece[m+n]{0}$ under $F_{n-1}^s$ for some $s$;
in this case $F^j(T)\cap \Gen[n-1]{J_F}\neq\emptyset$.
By \autoref{LmCopiesInWn}, $r-j$ is divisible by $2^{n-1}$. Moreover,
$F_{n-1}^p(F^j(T))=\Piece[m+n]{0}$ for $p=(r-j)/2^{n-1}$ and so we must have
$F^j(T)\subset X_{n,m}$.

\smallskip
Let $Q$ be the unique copy of $\Piece[n]{0}$ under $F^j$ containing $T$. Let
us show that $Q$ is either primitive or separated and
$Q\setminus\tilde{X}_{n+m}\neq\emptyset$.
Observe that if $Q\subset \tilde{X}_{n+m}$, the \refNestingProp\ and
\autoref{LmCopiesInWn} would imply that either $F^d(Q)\subset\Piece[n+m]{0}$
or $F^d(Q)\supset\Piece[n+m]{0}$  for some $d<j$. If $Q$ is primitive this is
clearly impossible.

Assume that $Q$ is not primitive. Then there exists $0\le\ell<j$
such that $F^\ell(Q)\subset\W[n]$. Assume that $\ell$ is the maximal such
number. If $F^\ell(Q)$ intersects $\Gen[n-1]{J_F}$ then from
\autoref{LmCopiesInWn} we obtain that $c=(j-\ell)/2^{n-1}$ is an integer and
$F_{n-1}^c(F^\ell(Q))=\Piece[n]{0}$. In this case
$F_{n-1}^c(F^\ell(T))=F^j(T)\subset X_{n,m}$ and so $F^\ell(T)\subset X_{n,m}$. This
contradicts the definitions of $j$ and $\ell$. Thus,
$F^\ell(Q)\cap\Gen[n-1]{J_F}=\emptyset$ and so $Q$ is separated.

Now, $F^d(Q)\supset\Piece[n+m]{0}$ is impossible for $d<j$ by
\autoref{LmSeparated} and the fact that $F^d(Q)$ is either a primitive or a separated copy of $P_0^{(n)}$ (see \autoref{RemSepCop}\ref{IterofSep}).
On the other hand, if  $F^d(Q)\subset\Piece[n+m]{0}$ for $d<j$ then
$F^d(T)\subset\W[n+m]\subset X_{n,m}$,  contradicting the definition of
$j$. This shows that $Q\setminus\tilde{X}_{n+m}$ is nonempty, finishing the
proof.
\end{proof}

\begin{lem}\label{LmUSQDisj}
For every primitive or separated copy $Q$ of
  $\Piece[n]{0}$  for which $Q\setminus \tilde{X}_{n+m}\neq \emptyset$, one has
\[ 
   Q\cap (\Sigma_{n,m}\cup X_{n,m})=\emptyset.
\]
\end{lem}

\begin{proof}
Assume that $Q\cap X_{n,m}\neq\emptyset$.
Then by the \refNestingProp, $Q$ either contains or is contained in
a copy $T$ of $\Piece[n+m]{0}$ under $F_{n-1}^{\ell}$ for some $\ell$.
If $T\subset Q$ then $Q$ intersects $J_F^{n-1}$, which is impossible since $Q$
is separated.
If, on the other hand, $Q \subset T$, then by
the definition of $\tilde{X}_{n+m}$  we have $Q\subset\tilde{X}_{n+m}$,
contradicting a hypothesis of the lemma.
Thus, we must have $Q\cap X_{n,m}=\emptyset$.
\smallskip

Now suppose that $Q\cap\Sigma_{n,m}$ is nonempty.
Let $k$ be such that $F^k(Q)= \Piece[n]{0}$, and let $\ell$ be the minimal number such that
$Q\cap F^{-\ell}_{n-1}(Y_n) \ne \emptyset$. Then
\[ F^s(Q)\setminus \W[n]\neq \emptyset\;\;\text{for all $s>2^{n-1}\ell$}, \]
and therefore $k\le 2^{n-1}\ell$.
By the \refNestingProp, 
\[ F^{2^{n-1}\ell}(Q) \;\supset\; \Piece[n]{0,L} \;\supset\;
   F_{n-1}^\ell(Q\cap F^{-\ell}_{n-1}(Y_n)) \]
for some quadrant $L$.
Since $Q\cap F_{n-1}^{-\ell}(\Piece[n]{0,L})$ is nonempty,
$Q$ must contain a copy $T$ of $\Piece[n]{0,L}$  with
$F_{n-1}^\ell(T) = \Piece[n]{0,L}$. Therefore $Q$ intersects $\Gen[n-1]{J_F}$,
contradicting \autoref{LmSeparated} and finishing the proof.
\end{proof}

\medskip
Let {$\SP$} be the set of all primitive or separated copies $Q$ of
$\Piece[n]{0}$ which lie in $\W[1]$ and are such that
$Q\setminus\tilde{X}_{n+m}\ne\emptyset$.
 For $Q\in\SP$ let $k$ be such that $F^k(Q)=\Piece[n]{0}$. Set
 $$\Sigma_Q=F^{-k}(\Sigma_{n,m})\cap Q,\quad X_Q=F^{-k}(X_{n,m})\cap Q.$$
Recall that $\Sigma_{n,m}$ and $X_{n,m}$ are symmetric with respect to the
axes. Observe that $F^k$ sends 
$X_Q$ and $\Sigma_Q$ bijectively onto the intersections of $X_{n,m}$ and
$Z_{n,m}$, respectively, with one of the four quadrants.


\begin{cor}\label{CoUSQDisj}
The sets from the collection $\Set{\Sigma_Q, X_Q \st Q\in\SP}$ are pairwise
disjoint.
\end{cor}

\begin{proof}
Since $\Sigma_{n,m}\cap X_{n,m}=\emptyset$, we obtain immediately that
$\Sigma_Q\cap X_Q=\emptyset$ for every $Q\in\SP$. Let $Q_1, Q_2\in\SP$ be
distinct copies of $T_1=\Piece[n]{0,K}$ and $T_2=\Piece[n]{0,L}$,
respectively, and let $k_1,k_2$ be such that $F^{k_1}(Q_1)=T_1$ and
$F^{k_2}(Q_2)=T_2$.

If $k_1=k_2$ then since $Q_1$ and $Q_2$ are distinct, they must be disjoint.
Assume that $k_1<k_2$ and that one of two sets $\Sigma_{Q_1}$,
$X_{Q_1}$ intersects one of two sets $\Sigma_{Q_2}$, $X_{Q_2}$.
By the nesting property (\autoref{NestingProperty}), $Q_2\subset Q_1$.
Since $F^{k_1}(Q_2)\in \SP$, we can apply \autoref{LmUSQDisj} to see that
$$ F^{k_1}(Q_2)\cap(\Sigma_{n,m}\cup X_{n,m})=\emptyset.$$
Therefore, $Q_2\cap (\Sigma_{Q_1}\cup X_{Q_1})=\emptyset$. This
contradiction finishes the proof.
\end{proof}
\goodbreak

\begin{proof}[\textbf{\textit{Proof of \autoref{ThRecEst}}}]
By \autoref{LmSeparatedness}, $\tilde{X}_{n+m}$ is the union of all
sets of the form $X_Q$, $Q\in\SP$, together with a countable set of
analytic curves. Using the fact that $X_{n,m}$ and $\Sigma_{n,m}$ are
symmetric with respect to the axes, \autoref{PropKoebeBA},
the definition of $M_{n,m}$~(\autoref{DefXnYn}), and
\autoref{ObsUnm}\ref{Etam+1}, we obtain
\[
  \frac{\ar(X_Q)}{\ar(\Sigma_Q)}  \le
  M_{n,m} \ar((\lambda^{-n}X_{n,m})\cap \Piece{0,\I}) =
  M_{n,m} \tilde\eta_{m+1} \ar(\Piece{0,\I}).
\]
We have
\[  \ar(\tilde{X}_{n+m})=\sum_{Q\in\SP} \ar(X_Q) \;\le\;
    M_{n,m}\tilde\eta_{m+1}\ar(\Piece{0,\I})\!\!\sum_{Q\in\SP}\ar(\Sigma_Q).
\]
Since
$\DS \big(\!\bigsqcup_{Q\in\SP} X_Q\big) \sqcup
   \big(\!\bigsqcup_{Q\in\SP} \Sigma_Q\big) \subset \tilde{X}_n$,
~{we obtain}~
$\DS \ar(\tilde{X}_{n+m}) \le
   M_{n,m}\tilde\eta_{m+1}\ar(\Piece{0,\I})\ar(\tilde{X}_n)$.
\end{proof}

\subsection{Estimating $\tilde{X}_n$ and $\Sigma_n$}
\newcommand{\PuzJn}{\Puz[n]_{\!\!J}}
Let $\PuzJn$ denote the elements of $\Puz[n]$ that intersect $J_F$.  That
is, $\PuzJn = \{P\in\Puz[n] \st P\cap J_F \ne \emptyset \}$.
Then, for each $n\in\N$, let $V_n$ be the interior of the closure of the
union of all tiles $P\in\PuzJn$.  Observe that the sets $V_n$ form a collection of nested neighborhoods of $J_F$.
\autoref{FigTilesWithJ} shows the tiles that make up $V_1$, $V_2$, and
$V_3$ in the first quadrant; $V_2$ is also shown in
\autoref{FigEscapingDisks}.

\begin{rem}\label{V2}
In what follows, $V_2$ shall be particularly useful. Observe that
$V_2=F^{-3}(\W).$

\noindent
Also, let
$\Vslit=(-\infty,-\tfrac{1}{\lambda}]\cup V_2 \cup [\tfrac{1}{\lambda^2},\infty)$.
\end{rem}

\begin{Def} \label{DefWtilde}
For $n\ge 3$, let $\Wtilde[n]$ denote the
interior of the closure of
the union of the copies $P$ of~$\mathbb{H}_\pm$ under $F^{2^n-6}$ with $0\in
\closure P$. Notice that for each $n\ge 3$ there are exactly four such
copies; denote by $\Ptilde[n,K]$ the copy in quadrant $K$ (if the quadrant
is omitted, we mean the appropriate copy).
\end{Def}

\begin{figure}[htbp]
 \centerline{
   \includegraphics[height=.23\hsize]{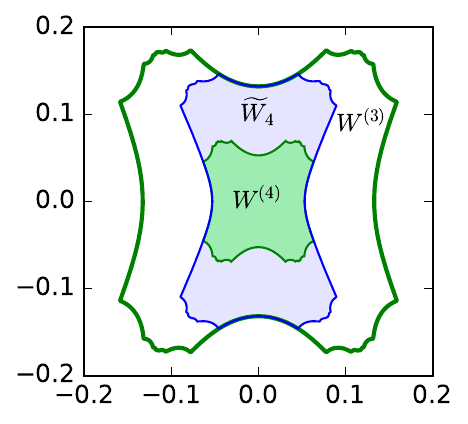}\hfil
   \includegraphics[height=.23\hsize]{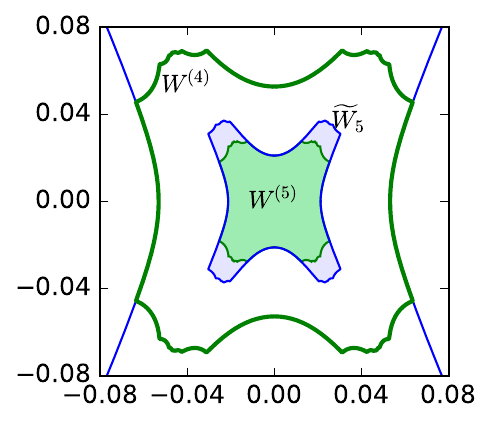}\hfil
   \includegraphics[height=.23\hsize]{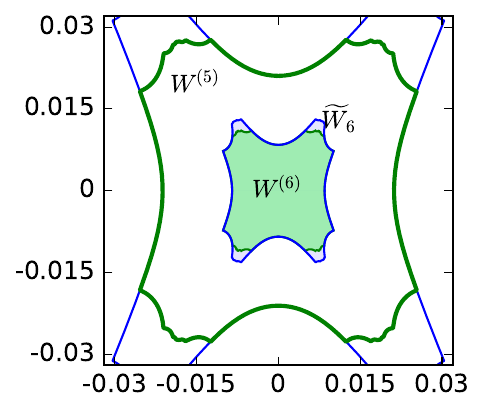}\hfil
   \includegraphics[height=.23\hsize]{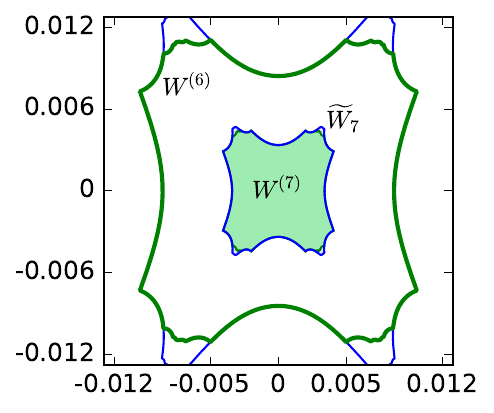}}
 \caption{\label{PicWtilde}
  The sets $\Wtilde[4]$, $\Wtilde[5]$, $\Wtilde[6]$, and $\Wtilde[7]$, shown
  along with  $\W[3]$, $\W[4]$, $\W[5]$, $\W[6]$ and $\W[7]$.}
\end{figure}

\goodbreak
\begin{rem}\label{WtildeProps}
The following observations are immediate from the definitions.
See \autoref{PicWtilde}.

\begin{enumerate}[itemsep=.15ex plus 2pt minus 2pt]
\item \label{WtildeBoundary} for all $n\ge3$, \quad
 $F^{2^n-6}(\partial\Wtilde[n]) = \R$\,,
\item \label{WtildeScale} for all $n\ge3$, \quad
 $\Wtilde[n+1] \;\subset\; \lambda\Wtilde[n]$\,,
\item \label{WtildeNest} for all $n\ge4$,\quad\,
 $ \W[n] \;\subset\; \Wtilde[n] \;\subset\; \W[n-1]$\,,
\item $\Wtilde[3]=\W[1]$\,,
\item \label{PtildeInduction}
 For $n\ge k\ge 3$,\; $F^{2^n-2^k}(\Ptilde[n])=\Ptilde[k]$\,.
\end{enumerate}
\end{rem}
%


\goodbreak
For a point $z_0 \not\in J_F$, the next lemma gives us explicit
criteria for determining a disk of points around $z_0$ whose orbits
behave comparably.
 See \autoref{FigEscapingDisks}.

\begin{lem}\label{LmEscape}
Let $D$ be a disk in the complement of $\Vslit$
and let $D_0$ be a connected component of $F^{-k}(D)$ for any $k\ge 0$.
Then for $n\ge3$, either $D_0 \cap \W[n]=\emptyset$ or $D_0 \subset \Wtilde[n]$.
\end{lem}

\begin{figure}[htb]
\centerline{\includegraphics[height=.35\textwidth]{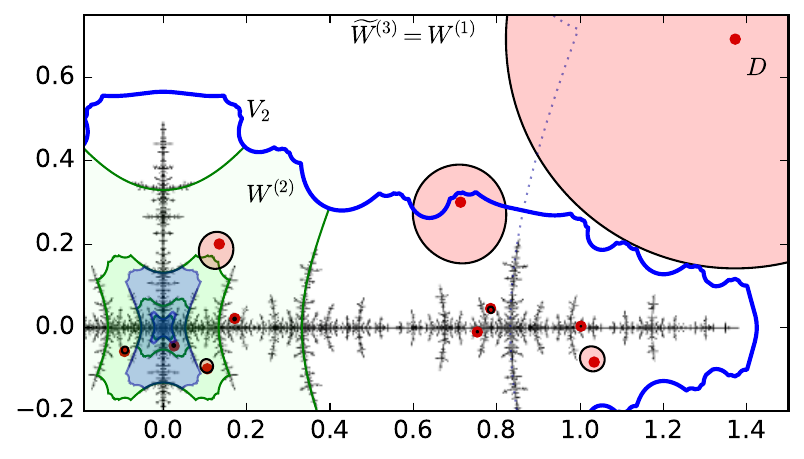} \quad
            \includegraphics[height=.35\textwidth]{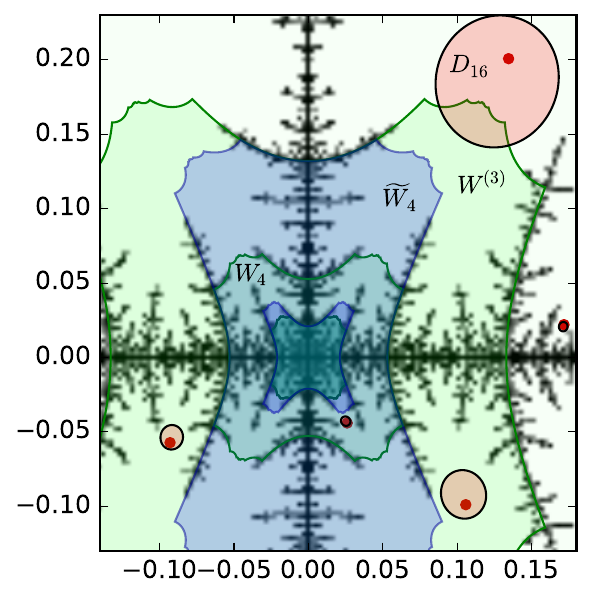}}
\caption{\label{FigEscapingDisks}
  Illustration of \autoref{LmEscape}: on the left is shown $V_2$ (blue
  curve) and a large disk $D$ which lies outside of $\Vslit$.   Also shown are
  $\W[2]$, $\W[3]$, $\W[4]$ and $\W[5]$ (in green), as well as $\Wtilde[4]$ and
  $\Wtilde[5]$ (shaded blue) and $J_F$.
  Several preimages of $D$ are drawn in light red, with the
  corresponding preimage of the center of $D$ indicated by a red dot.
  On the right is a blow-up of the left figure.  While the preimage labeled
  $D_{16}$ partially intersects $\W[3]$, it lies completely inside
  $\Wtilde[3]=\W[1]$ (dotted line).
}
\end{figure}

\goodbreak
\begin{proof}
If we assume the lemma does not hold, $D_0$ must contain points from
$\partial\W[n]$ and $\partial\Wtilde[n]$.

First, since $V_2$ is a neighborhood of $J_F$, the definition of $D$
means it can contain no points of $J_F$.
Using $D_j$ to denote $F^j(D_0)$, we have $D_j\cap J_F=\emptyset$ for
$0\le j\le k$.

Observe also that $\W[n]\subset F^{-(2^n-4)}(V_2)$, and hence
$k\ge 2^n-3$. Applying \autoref{WtildeProps}\ref{WtildeBoundary} gives
\begin{equation}\label{EqFkWtildeReal}
 F^k(\partial\Wtilde[n])=F^{k-2^n+6}(F^{2^n-6}(\partial\Wtilde[n]))
   =F^{k-2^n+6}(\R) = \R,
\end{equation}
so we must have $D_k \cap \R \ne\emptyset$.

Further, if for some $j<k$ we have $D_j$ intersecting both $\R$ and $\i\R$,
it can be shown by induction that
$D=D_k$ must contain both positive and negative real values, which is
impossible.

If $D_j\cap\i\R=\emptyset$ for all $j<k$,
then since $D_j$ can contain no points of $J_F$, we must
have
\[
D_j\subset\overline{%
  \Piece[1]{1,\I}\cup\Piece[1]{1,\II}\cup\Piece[1]{1,\III}\cup\Piece[1]{1,\IV}}
  \quad\text{for all $j<k-1$.}
  \]
This contradicts our initial hypothesis that $D_0$ intersects $\W[n]$;
so for some $j<k$, we must have $D_j\cap\i\R \neq \emptyset$.
Let $s$ be the maximal index for which $D_s$ intersects the
imaginary axis,  and let $x$ be a point in $D_s\cap\i\R$. Without loss
of generality, we may take $\Im(x)>0$.

As noted earlier, $D_s$ cannot intersect both $\R$ and $i\R$.
Hence $D_s\cap \R=\emptyset$.
Because $D_s$ contains a boundary point of $F^{s}(\Wtilde[n])$, by
\autoref{WtildeProps}\ref{WtildeBoundary} we
must have $s\le 2^n-7$.
Combining this with the fact that $k\ge 2^n-3$ yields $s\le k-4$.

\smallskip
\newcommand{\VR}{V_{2,\R}}
Let $\VR$ be the union of closures of tiles from $\Puz[2]$ which
intersect $J_F\cap\R$.
\autoref{LmImCrit} implies that
$\VR\cap\i\R=[-\lambda y_0,\lambda y_0]\subset J_F$.
Thus $D_s$ intersects $\i\R$ outside $\VR$, so
$\pm\lambda y_0 \not\in D_s$.
Since $D_s$ intersects $F^{s}(\W[n])$ and hence also intersects $\VR$,
we conclude that $D_s$ contains a boundary point of $\VR$.

But since $k-s\ge 4$ and
$F^4(\partial\VR \cup [\lambda y_0,x])\subset \R$,
it follows that $D_k\cap\R$ consists of at least two connected components;
this is impossible.
The contradiction finishes the proof.
\end{proof}

\medskip
Use $\Disk[R]{z}$ to denote the open disk of radius $R$ centered at $z$, and
recall the definition of $\H[1]$ from \autoref{DefWandH}.
From \autoref{DefXnYn}, recall that $\tilde{X}_n$ is the set of points of
$\W[1]$ that eventually land in $\W[n]$ under iterates of $F$, $Y_n$ are
points of $\W[n]$ which never return to $\W[n]$ under
non-trivial iterates of $F$ and $\Sigma_n$ is the set of points in $\W[n]$
that eventually land in $Y_n$ under iterates of $F_{n-1}$.

Applying the Koebe One-Quarter Theorem together with \autoref{LmEscape}
yields the following two useful corollaries, which enable us to estimate the
size of disks which lie outside $\tilde{X}_n$ or inside $\Sigma_n$.

\begin{cor}\label{CoCalcXn} Fix $n\ge 3$. Let $z\in \H[1]\setminus J_F$ and let $k$ be
  such that $F^k(z)\notin \Vslit$. Suppose also that $F^j(z)\notin \Wtilde[n]$
  for  $0\le j\le k$.  Let $D_0$ be the connected component of
  $F^{-k}\bigl(\Disk[R]{F^k(z)}\bigr)$ containing~$z$,
  with 
  $R=\dist(F^k(z), \Vslit)$.
  Then $D_0\cap\tilde{X}_n=\emptyset$. In particular,
  \[
    \Disk[r]{z}\cap \tilde{X}_n=\emptyset,\;\;\text{where}\;\;
    r=\frac{R}{4|DF^k(z)|}.
  \]
\end{cor}

\begin{cor}\label{CoCalcSigma} Fix $n\ge 3$. Let $z\in \W[n]$ be such that
  $w=F_{n-1}^s(z)\in Y_n$ for some $s$;
  let $\ell$ be such that $F^\ell(w)\notin \Vslit$.
  Suppose also that $F^j(w)\notin \Wtilde[n]$ for all $0\le j\le \ell$.

  Set 
  $R=\dist(F^\ell(w), \Vslit)$
  and let $D_0$ be the connected component of
  $F_{n-1}^{-s}\Bigl(F^{-\ell}\left(\Disk[R]{F^\ell(w)}\right)\Bigr)$ that
  contains~$z$.
  Then $D_0\subset\Sigma_n$. In particular,
\[
 \Disk[r]{z}\subset\Sigma_n, \;\;\text{where}\;\;
  r=\frac{R}{4|DF^k(z)|} \;\;\text{and}\;\; k=2^{n-1}s+\ell.\]
\end{cor}

\smallskip
With \autoref{CoCalcXn} and \autoref{CoCalcSigma} in hand, we have
explicit, computable criteria for verifying that the hypotheses of
\autoref{PropRecMeas} hold, showing that the Hausdorff dimension
of $J_F$ is less than two.

\medskip
Specifically, for some $n$ we need to establish upper bounds on the
quantities $\tilde\eta_n$ and $M_n$.
We now give algorithms to do this.  These are presented assuming that
$F(z)$, $F'(z)$, $\Wtilde[n]$, $\Piece{0}$, etc.\
can be calculated exactly. In \autoref{SecComputation}, we discuss
how to account for finite precision.

\medskip
To bound $\tilde\eta_n$, we need an upper bound on $\ar(\tilde{X}_n)$,
since, $\tilde\eta_n = \ar(\tilde{X}_n)/\ar(\W[1])$ by definition.
Exploiting the symmetry with respect to the axes allows us to work in
first quadrant only.

\begin{algorithm}\label{AlgXn}
Fix $n\ge 3$ and $r$ small.
To compute an upper bound for $\ar(\tilde{X}_n)$, we find a collection
$\DD_{n,r}$ of disks with radius $r$ which cover  $\tilde{X}_n$ in the
first quadrant.
First, select a grid of points $z$ so that
  $\bigcup_z \Disk[r]{z}$ covers $\Piece[1]{0}$.
  Also, choose some upper bound $K$ on the maximum number of
  iterations. For each point $z$ in the grid run the following routine.
  \begin{enumerate}
    \item \label{XnIter} For $0< j<K$,  calculate $F^j(z)$ and $DF^j(z)$.
    \item\label{XnNoescape}
      If $F^j(z)\in\Wtilde[n]$ for some $j$, or if $F^j(z)\in \H[1]$
      for all $j<K$, add $\Disk[r]{z}$ to $\DD_{n,r}$, and exit the routine.
    \item\label{XnCheck}
      Let $\ell$ be such that  $F^\ell(z)\not\in\H[1]$.
      If $\dist(F^\ell(z),\Vslit)>4r\,|DF^\ell(z)|$, then by
      \autoref{CoCalcXn}, the entire disk $\Disk[r]{z}$ does not
      intersect $\tilde{X}_n$.
      Otherwise, add $\Disk[r]{z}$ to $\DD_{n,r}$.
  \end{enumerate}
\end{algorithm}

\begin{rem}
The cover $\DD_{n,r}$ can be calculated from $\DD_{n,s}$ for $r<s$ by
replacing the grid covering $\Piece[1]{0}$ by one covering
$\DD_{n,s}$.
\end{rem}

\medskip
In order to bound $M_n = M(\Sigma_n)$ from above, we need to construct a
lower bound on $\Sigma_n$ (\ie a subset of $\Sigma_n$).  As before, we can
exploit symmetry and work only in the first quadrant.

\begin{algorithm}\label{AlgSigma}
Fix $n\ge 3$ and $r$ small. To get a lower bound for $\Sigma_n$, we
find a collection $\EE_{n,r}$ of disks $\Disk[r]{z} \subset
\lambda^{-n}\Sigma_n$. As in \autoref{AlgXn}, select a grid of points $z$
  so that  $\Piece{0} \subset \bigcup_z \Disk[r]{z}$,
  and fix a positive integer $K$. For each point $z$ from the grid run the
  following routine.
\begin{enumerate}
  \item \label{SigmaIterz}
    For $0< j<K$ calculate $F^j(z)$ and $DF^j(z)$.
  \item\label{SigmaNoEscapeH1}
    If $F^j(z)\in \H[1]$
    for all $j<K$, discard $z$ and exit the routine.
  \item\label{SigmaEscapeH1}
    Let $0\leq k<K$ be the smallest number such that
    $F^k(z)\not\in\H[1]$. Set $w=\lambda^nF^k(z).$
  \item \label{SigmaIterw}
    For $0< j< K$,  compute $F^{j}(w)$ and $DF^{j}(w)$.
  \item \label{SigmaNoEscape}
    If $F^j(w)\in\Wtilde[n]$ for some $j$, or if $F^j(w)\in \H[1]$
    for all $j$, discard $z$ and exit the routine.
  \item\label{SigmaCheck}
    Let $\ell$ be such that  $F^\ell(w)\not\in\H[1]$.
    If $\dist(F^\ell(w),\Vslit)>4r\,|DF^\ell(w)|\cdot|DF^k(z)|$,
     then by \autoref{CoCalcSigma}, the entire disk
      $\Disk[r\lambda^n]{z}$ is contained in $\Sigma_n$;
      add the disk to $\EE_{n,r}$.
  \end{enumerate}
\end{algorithm}

\begin{rem}\label{RemFSpeedUp}
In step~\ref{XnIter} of \autoref{AlgXn} and steps~\ref{SigmaIterz}
and~\ref{SigmaIterw} of \autoref{AlgSigma}, we need not (and \emph{should
  not}) compute $F^j(z)$  for all $j<K$.
Instead, we restrict our attention to iterates $z_k=F^{j_k}(z)$
defined inductively as follows. Let $j_0=0$.
Assuming $z_k\neq 0$, there is a maximal number $m_k$ such that $z_k\in\W[m_k]$;
let $i_k=\max\{0,m_k-1\}$.
Set $j_{k+1}=j_k+2^{i_k}$, and calculate
$z_{k+1}=F^{2^{i_k}}(z_k)$ and
$|DF^{j_{k+1}}(z)|=|DF^{j_k}(z)|\cdot|DF^{2^{i_k}}(z_k)|$ via~\eqref{EqF2m}:
\[
  F^{2^{i_k}}(z_k)=(-\lambda)^{i_k}F({z_k}/{\lambda^{i_k}}),
   \qquad
  |DF^{2^{i_k}}(z_k)|= |DF({z_k}/{\lambda^{i_k}})|.
\]
Observe that $F^{2^{i_k}}$ is the first return map from $\W[m_k]$ to
$\W[m_k-1]$ (see \autoref{RemCritOrb}\ref{RemIterW}).
Given $n\geq 3$ we have $\Wtilde[n]\subset \W[n-2]$.
Hence, for $z_k\notin\Wtilde[n]$, since $m_k\leq n-1$, it is not possible to have $F^l(z_k)\in
\Wtilde[n]$ for any $0<l<2^{i_k}$.
Also, it can be shown that if $z_k\in \H[1]$, we must have $F^l(z_k)\in \H[1]$ for $0<l<2^{i_k}$.

Also, we should consider all $k<K$ rather than $j<K$, bounding
the number of evaluations of $F^{2^{i_k}}$ rather than the length of the orbit
of $z$.

\end{rem}


\medskip
Using these algorithms, we complete the proof of our \refMainThm, computing
that
\[
 M_6=M(\Sigma_6) < 9.4,       \qquad
 \tilde{\eta}_6=\frac{\ar(\tilde{X}_6\cap\Piece[1]{0,\I})}{\ar(\Piece[1]{0,\I})}
<\frac{0.09}{\ar(\Piece[]{0})}.
\]
We obtain $\tilde\eta_6 M_6\ar(\Piece[]{0})< 0.846<1$, so $J_F$ has Hausdorff
dimension less than 2.

\section{Computational considerations}\label{SecComputation}

In this section, we discuss how we can be certain that the computation of
the bounds on $\tilde{\eta}_6$, $M(\Sigma_6)$, and $|\Piece{0}|$ (which
give us our \refMainThm) are sufficiently accurate, even though the bounds
are necessarily computed with finite precision.  We have two potential
sources of error: we can not know the map $F$ exactly (and
consequently we must also approximate $\lambda$ and the domains
$\W[n]$, $\H$, etc.), and there will be some errors introduced by the
approximation of exact quantities $z$ by those representable on a
computer.

\smallskip
The calculations for this paper were done primarily in Python with
double-precision floating-point arithmetic satisfying the
IEEE~754-2008 standard \cite{IEEE754}.  In this context, only numbers
of the form $\alpha \times 2^e$ are representable, where $\alpha$ is a
53-bit signed (binary) integer, and the exponent $e$ satisfies
$-1021 \le e \le 1024$.
In particular, the only real numbers which can be
represented exactly are certain dyadic rationals within a (large) range.
Any real number $x$ in the representable range can be approximated by a
floating-point number $\FLT{x}$ so that $|x-\FLT{x}|< \uflt|x|$.
This number $\uflt$ is called the \defn{unit roundoff}; for IEEE
double-precision
\[ \uflt \,=\, 2^{-53} \,\approx\, 1.11\times 10^{-16}. \]
For more details, the reader is referred to \cite{Higham} or
\cite{Goldberg}, for example.

While the standard guarantees that the result of a \emph{single} arithmetic
operation ($+$, $-$, $*$, $/$) carried out on two (real) floating-point
numbers will be correctly rounded with a relative error of at most $\uflt$,
we need to ensure that these small errors do not accumulate such that we lose
control of the calculation. We primarily need to work with double-precision
complex numbers;  in \cite{BrentPercivalZimmermann}, it is shown that the
relative error for complex arithmetic is bounded by $\sqrt{5}\,\uflt$.

\smallskip
Throughout this section, we shall use the notation $\FLT{x}$ to denote the
approximation of the exact quantity $x$ by one that is representable
as a floating point number.

\subsection{Approximating the map $F$ and the value of $\lambda$.}

In the 1980s, Lanford \cite{Lanford} calculated a high-precision
approximation of $F$ as an even polynomial of degree~80.  Such
approximations can be computed to precision $10^{-n}$ in
a number of arithmetic operations polynomial in $n$ \cite{FeigPolyTime},
although the approximation given by Lanford is sufficient for our purposes.

Lanford gives strict error bounds on his approximation (which he calls
$g_n^{(0)}$ but we refer to as~$\apF$ for notational
consistency). Specifically, we have the following.

\begin{prop}\label{LanfordsBounds}
Let $\apF$ be the degree~80 polynomial approximation
of $F$ from \cite{Lanford}.  The following upper bounds on the error apply.

\centering\begin{tabular}{ll c@{\quad} ll c@{\quad} ll}
\multicolumn{2}{c}{$|F(z)-\apF(z)|$}    &&
\multicolumn{2}{c}{$|F'(z)-\FLT{F'}(z)|$}  &&
\multicolumn{2}{c}{$|F''(z)-\FLT{F''}(z)|$} \\
\cline{1-2} \cline{4-5} \cline{7-8} \\[-2ex]
$1.5 \times 10^{-23}$ & for $|z|<1.224$ && 
$1.5 \times 10^{-22}$ & for $|z|<1.12$  &&
$1.5 \times 10^{-21}$ & for $|z|<1.02$  \\
$5.5 \times 10^{-13}$ & for $|z|<1.414$ && 
$5.5 \times 10^{-12}$ & for $|z|<1.31$  &&
$5.5 \times 10^{-11}$ & for $|z|<1.21$ \\
$5.0 \times 10^{-7}$ & for $|z|<2.449$ && 
$5.0 \times 10^{-6}$ & for $|z|<2.34$  &&
$5.0 \times 10^{-5}$ & for $|z|<2.24$ \\
$1.7 \times 10^{-2}$ & for $|z|<2.828$ && 
$1.7 \times 10^{-1}$ & for $|z|<2.72$ &&
$1.7              $ & for $|z|<2.62$ \\
\end{tabular}
\end{prop}

\begin{proof}
The bounds for $|F(z)-\apF(z)|$ are taken from \cite{Lanford}.
Those for the first and second derivatives follow Lanford's bounds
via an application of Cauchy's derivative inequalites.
\end{proof}

\begin{cor}\label{CorLambda} All of the digits in the approximations
\[ \lambda \approx 0.39953\;52805\;23134\;48985\;75
   \qquad
   1/\lambda \approx 2.50290\;78750\;95892\;82228\;39
\]
are correct (the next two digits of $\lambda$ are between $65$ and $96$).
\end{cor}

If $\apF$ is evaluated at some $\apz$ using Horner's method, we have
the following bound on the accumulated arithmetic error in evaluating the
polynomial:
\[
 |\apF(z) - \apF(\apz)| <
 \sqrt{5}\,\frac{80\uflt}{1-80\uflt}\sum_{i=0}^{40} |a_i||\apz|^{2i},
\]
where $a_i$ are the coefficients of $\apF$
(see \cite[ch.~5]{Higham} which presents the argument from
\cite{Wilkinson}); a similar bound applies to the derivative.
Putting this observation together with
\autoref{LanfordsBounds} gives us the following.

\begin{cor}\label{PropEvalFError}
If $\apz$ is a double-precision approximation of $z$ and $\apF$
is Lanford's polynomial approximation of $F$ evaluated in
double-precision, then we have the following upper bounds on the error.

\centering\begin{tabular}{ll c@{\qquad} ll}
\multicolumn{2}{c}{$|F(z)-\apF(\apz)|$}    &&
\multicolumn{2}{c}{$|F'(z)-\FLT{F'}(\apz)|$} \\
\cline{1-2} \cline{4-5} \\[-2ex]
$5.2884 \times 10^{-14}$ & for $|z|<1$ &&
$7.2771 \times 10^{-14}$ & for $|z|<1$  \\
$6.4430 \times 10^{-13}$ & for $|z|<1.414$ && 
$5.5875 \times 10^{-12}$ & for $|z|<1.31$ \\
$5.0001 \times 10^{-7}$ & for $|z|<2.449$ && 
$5.0001 \times 10^{-5}$ & for $|z|<2.34$ \\
$1.7001 \times 10^{-2}$ & for $|z|<2.828$ && 
$1.7001 \times 10^{-1}$ & for $|z|<2.72$
\end{tabular}
\end{cor}

Observe that for $|z|<1$, the error in using $\apF(\apz)$ is dominated
by the accumulated round-offs (since $F$ is approximated by $\apF$
to better than machine precision for $|z|<\sqrt{6}$);
for $|z|>\sqrt{2}$, the error is
dominated by the approximation of $F$ by $\apF$.

\begin{rem}\label{RemSpeedErrs}
As noted in \autoref{RemFSpeedUp}, we will often want to evaluate
$F^{2^n}\!(z)$ as $(-\lambda)^n F(z/\lambda^n)$.
Since \autoref{CorLambda} gives $\lambda$ and $1/\lambda$ with greater
precision than can be  represented as a floating-point number, we can
calculate floating-point approximations of $\lambda^n$ and
$\lambda^{-n}$ with a relative error of no more than
$\frac{n\uflt}{1-n\uflt}<1.0102n\uflt < 1.13n\times 10^{-16}$.
\end{rem}

\begin{rem}\label{RemCompositionErrs}
By the Koebe Distortion Theorem, if $g\from\Disk[|z|]{0} \to \C$
is univalent and $|\epsilon|<|z|$, then
\[
|g(\epsilon) - g(0)| \le
  \frac{|\epsilon|}{(1-|\epsilon|/|z|)^2}\cdot |g'(z)|.
\]
This observation will be useful in calculating accumulated error bounds for
compositions.
\end{rem}


\subsection{Approximating the sets $\W$, $\H[1]$, $V_2$, and $\Wtilde[6]$}

\begin{Def}
For complex numbers $a$ and $b$, let
\[ \CBOX{a}{b} = \Set{z \st \Re(a) < \Re(z) < \Re(b), \quad
                            \Im(a) < \Im(z) < \Im(b)}, \]
that is, the open rectangle in $\C$ with $a$ on its lower left corner and $b$
on its upper right.

Also, for a set $A$ in the first quadrant, let $\ICL(A)$ be the
interior of the closure of the union of the four symmetric copies of
$A$, that is
\[ \ICL(A) = \inter \closure{A \cup -A \cup A^* \cup -A^*}, \]
where $A^*$ is the set of complex conjugates of points of $A$,
and $-A$ is the negation of all points in $A$.
\end{Def}

\begin{Def}\label{DefHoutHin}
Define the following sets of floating-point numbers.
\begin{align*}
  \Pin  = \CBOX{0}{2.07+2.06i}  \quad&\text{and}\quad
  \Pout = \CBOX{0}{2.495+2.81i}, \\
  \Qin  = \CBOX{2.07}{3.75+1.65i} \quad&\text{and}\quad
  \Qout = \CBOX{2.495}{4.25+2.35i}.\\[-.5\baselineskip]
  \intertext{Then let}
  \phantom{x}\\[-1.6\baselineskip]
  \Win  = \ICL(\Pin) \quad&\text{and}\quad
  \Wout = \ICL(\Pout), \\
  \Hin  = \ICL(\Pin \cup \Qin) \quad&\text{and}\quad
  \Hout = \ICL(\Pout\cup \Qout). \\[-.5\baselineskip]
\intertext{Let $\Vout$ be the filled polygon with  vertices at\;
  $0$, ${4.075i}$, ${9.33+.85i}$, and ${9.33}$; set
}
  \Vtwo &= \ICL\left(\Vout[2]\right).\\[-.5\baselineskip]
 \intertext{Finally, set}
 \apPtilde_6 =\Gen[6]{(\Pout\cup\CBOX{1.30+2.81i}{2.2+3.23i})}
 \quad&\text{and}\quad
 \apWtilde_6 = \ICL(\apPtilde_6).
\end{align*}
\end{Def}

\begin{figure}[thb]
 \centerline{\includegraphics[height=9\baselineskip]{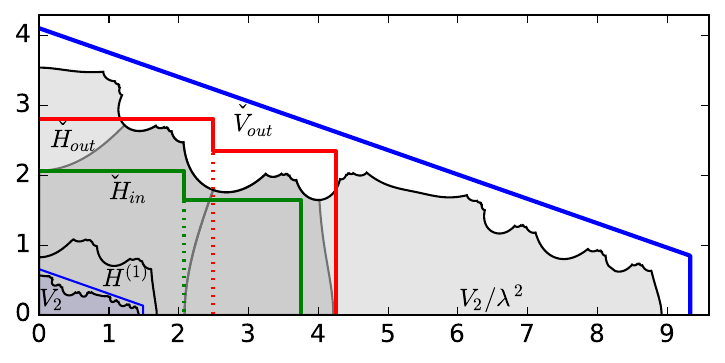} \quad
             \includegraphics[height=9\baselineskip]{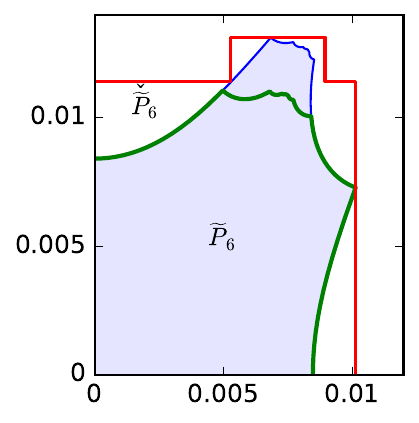}}
 \caption{\label{FigHin}
   As in \autoref{DefHoutHin}, the sets
   $\Hout$~(in red), $\Hin$~(green), and $\Vout$~(blue) are shown in the first
   quadrant, along with
   $\H$ (shaded gray) and $V_2$ scaled by a factor of $\lambda^2$ (also
   shaded).
   On the right are $\Ptilde[6]$ (shaded blue), $\Piece[6]{0}$~(outlined in
   green) and $\apPtilde_6$ (outlined in red).
 }
\end{figure}

\begin{rem}\label{RemFlambda}
As should be apparent from the figures, it is possible to obtain much better
approximations of the relevant sets than given in \autoref{DefHoutHin}.
To do so, one  can exploit the fact that the set ${\Piece{0,\I}}$ is fixed
under the map $F_{\lambda}(z)=\overline{F(\lambda z)}$ (see \cite{Buff}).
$F_{\lambda}$ has an attracting fixed point on
$\partial{\Piece{0,\I}}$ at $x_0$, and a repelling fixed point at
$x_1\approx 1.831+2.683i$.  This point $x_1$ is the unique point in
$\partial{\Piece{0,\I}}\cap\partial\widehat{W}$.
Taking repeated preimages of the segment $[0,\lambda y_0]$ by the map
$\apF_\lambda$ yields a good approximation of $\Piece{0,\I}$;
other pieces of $\Puz$ can then be approximated via preimages of $\Piece{0}$.
However, the polygonal sets of \autoref{DefHoutHin} are much easier to
obtain sharp error bounds on, and are sufficent for our purposes.
\end{rem}
\goodbreak

\begin{lem}\label{LmWapprox} One has
$\Win\subset \W \subset \Wout$, \;
$\Hin[1]\subset \H[1] \subset \Hout[1]$, \;
$V_2 \subset \Vtwo$,  \;and\;
$\Wtilde[6] \subset  \apWtilde_6$.
\end{lem}

\begin{proof}
To see that $\Win \subset \W$, compute the approximation of
$F(\partial\Win)$.  For $z\in \partial\Win$ with $|z|<\sqrt{6}$ we use
$\apF(\apz)$.
For $|z|\ge\sqrt{6}$, we instead approximate $F(z)$ by
$(-1/\lambda)\apF^2(\lambda\apz)$ (since we don't have a bound on
$|F(z)-\apF(\apz)|$ for $|z|\ge\sqrt{8}$).
Observe that all points of the approximation of $F(\partial \Win)$ avoid the
slits $(-\infty, -1/\lambda]$ and $[1/\lambda^2,\infty)$ by
  at least $5.4\times10^{-4}$ for $|\apz|<\sqrt{6}$ and more than
  $2.8$ for $|\apz|\ge\sqrt{6}$,
significantly more than the error bounds given in \autoref{PropEvalFError}.
Since
$F(\W) = \C\setminus\bigl( (-\infty,-1/\lambda]\cup[1/\lambda^2,\infty)\bigr)$,
we have $\Win \subset \W$.

\medskip
A similar calculation shows that $\Hin[1] \subset\H[1]$:
$\apF\big(\Qin[1]\big)\subset\Win$ with a margin of
more than~$.004$, which by \autoref{PropEvalFError} is much greater than
the necessary space since $|z|<1.75$ for $z\in\Qin[1]$.
Since $\Hin=\Pin\cup\Qin$, we know $F(\Hin[1])\subset\W$ and
thus $\Hin[1]\subset\H[1]$.

\smallskip
To see that $V_2 \subset \Vtwo$, recall from \autoref{V2} that
$V_2=F^{-3}(\W).$   We need only
check that $\apF^3(\partial\Vtwo)$ stays away from $\Wout$  with a sufficient margin:  the distance between the boundaries of
these two sets is greater than $1.5\times 10^{-3}$ (with the closest
point of $\apF^3(\partial\Vtwo)$ being the image of $9.33\lambda^2$).

We calculate that $|\apF^3(\apz) - F^3(z)| < 4\times 10^{-5}$ on
$\Vtwo$ by observing that for $\apz\in \Vtwo$, we have
$|\apz|<1.50$,
$|\apF(\apz)|  <1.71$ and
$|\apF^2(\apz)|<2.35$,  with
$|\apF'(\apz)|        < 2.81$,
$|\apF'(\apF(\apz))|  < 3.88$ and
$|\apF'(\apF^2(\apz))|<19.27$.
Applying the Koebe Distortion Theorem (see
\autoref{RemCompositionErrs}) twice gives an upper bound on the error of
$3.8\times 10^{-5}$.
Once we establish that $\W \subset \Wout$, the desired result follows.

\medskip
Since $\Wout$ contains points outside the domain of definition
of $F$, we cannot verify directly that $\W \subset \Wout$.
Instead, to do this we use the properties of the map
$F_\lambda(z)=\overline{F(\lambda z)}$
mentioned in \autoref{RemFlambda}.
Notice that $F_\lambda(z)$ can be approximated on $\Wout$ with an error less
than $5.1\times 10^{-7}$ since $|\lambda z|<1.502$ on $\Wout$ (see
\autoref{PropEvalFError}).

 Approximating $F_\lambda(z)$ by $\apF_\lambda(\apz)=\overline{\apF(\lambda\apz)}$ we observe that
$\apF_\lambda(\partial\Pout)$ almost avoids $\Pout$, intersecting it in
two curves lying inside the triangles
$T_1 = \triangle(2.08i,\; 2.81i,\; .31+2.81i)$ and
$T_2=  \triangle(2.495,\; 2.05+2.81i,\; 2.495+2.81i)$.
Although the triangles $T_1$ and $T_2$ lie inside $\Pout$,
we have $\dist(\apF_\lambda^2(T_1),\,\Pout) > .01$ and
$\dist(\apF_\lambda^2(T_2),\,\Pout) > .22$.
A straightforward computation shows that for the relevant points we
have $|F^2_\lambda(z) - \apF_\lambda^2(\apz)| < 1.7\times 10^{-6}$,
so $F^2(T_1)$ and $F^2(T_2)$ lie outside of $\Pout$. This implies that
$\Piece[]{0,\I}\subset\Pout$ (and so $\W \subset \Wout$
and $V_2\subset\Vtwo$). See \autoref{FigH1covers}.

\begin{figure}[bht]
 \centerline{\includegraphics[height=9\baselineskip]{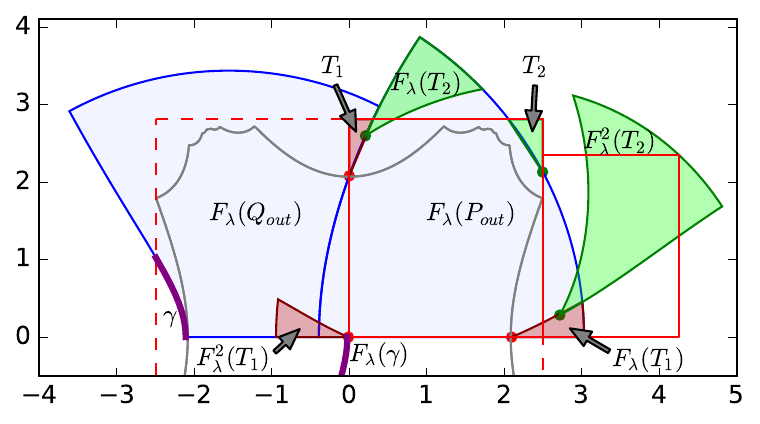} \quad
             \includegraphics[height=9\baselineskip]{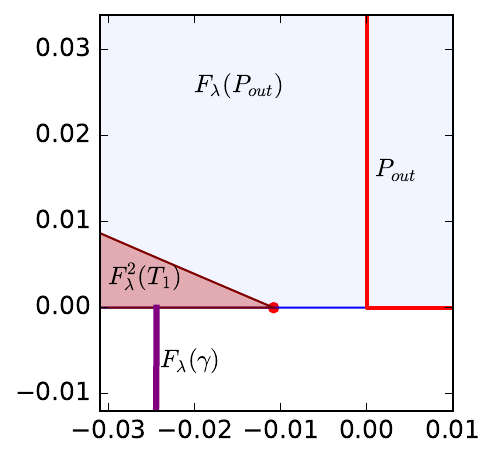}}
 \caption{\label{FigH1covers}
   Showing that $\W\subset \Wout$ and $\Hin[1]\subset\H[1]$.
   $F_\lambda(\Pout\cup\Qout)$ is shaded in blue, with $\Pout$ and $\Qout$
   outlined in red. The image on the right is a zoom of the one on the left.
   }
\end{figure}

Indeed, assume that $\Piece{0,\I}\not\subset\Pout$, and let
$\Gamma$ be the closure of the part of $\partial\Pout$ in the first quadrant
($\Gamma=[2.81i,\ 2.495+2.81i]\cup[2.495+2.81i,\ 2.495]$).
Then $\partial\Piece{0,\I}$ must intersect $\Gamma$. Recall that
$\partial\Piece[]{0,\I}$ is $F_\lambda$-invariant with two fixed points $x_0$
(attracting) and $x_1$ (repelling); the two components of
$\partial\Piece{0,\I} \setminus\Set{x_0,x_1}$ are interchanged by $F_\lambda$.
Parameterize the two components of $\partial\Piece{0,\I}
\setminus\Set{x_0,x_1}$ so that the parameterizations are the two inverse
branches of a function
$p:\partial\Piece{0,\I}\setminus\{x_1\}\to [0,\infty)$
satisfying%
\footnote{One way to construct such a function $p$ is to
  set $p(x_0)=0$, and for $z\in[0,F(\lambda)]$ let $p(z)=4-3z/F(\lambda)$.
  Then  propagate~$p(z)$ to the rest of
  $\Piece{0,\I}\setminus\{x_1\}$ via $p(F_\lambda(z))=\tfrac{1}{2}p(z)$.
}\
$p(F_\lambda(z))=\tfrac{1}{2}p(z)$ for every $z$ and $p(x_0)=0$.
Consider the closed set $\partial\Piece[]{0,\I}\cap\Gamma$; let $z_1$ be the
point where $p(z_1)$ attains the minimum on this set.
As a consequence of the previous argument, either $F_\lambda(z_1)$ or
$F_\lambda^3(z_1)$ lies outside $\Pout$ (see \autoref{FigH1covers}). Therefore,
there exists a point $z_2\in\Piece{0,\I}\cap\Gamma$ with
$p(z_2)\leq\tfrac{1}{2}p(z_1)$. This contradicts the definition of $z_1$, and
hence $\Piece{0,\I}\subset\Pout$.

\medskip
To show that $\H[1] \subset  \Hout[1]$ it is sufficient to confirm that
$F_\lambda(\partial\Hout)\cap (\W\cap\mathbb H_+)=\emptyset$.   Observe that
$F_\lambda(\partial\Hout)\cap (\Wout\cap\mathbb H_+)$ consists of
two curves, one of which belongs to $T_2$ and so is outside of~$\W$. The
other curve $\gamma$ joins the leftmost boundary of $\Wout$ and the real axis.
See \autoref{FigH1covers}.  This curve $\gamma$ must lie outside of $\W$, since
$\re\apF_\lambda(\apz)<- 0.024$ for all $z\in\gamma$ and
$\re(z)>0$ for $z\in F_\lambda(\W)$.
The error in computing
$\apF_\lambda^2$ on all of $\Hout$ is less than $2.3\times10^{-6}$, giving
a sufficient margin of error.

\medskip
In order to prove that $\Wtilde[6] \subset  \apWtilde_6$, recall from
\autoref{WtildeProps}\ref{PtildeInduction} that
$F^{2^n - 2^k}(\Ptilde[n]) = \Ptilde[k]$ for all $n\ge k\ge3$.
Hence $F^{54}(\Ptilde[6]) = F^{6}(\Ptilde[4])=\Piece[2]{0}$.
Observe that the curve $\apF^{54}(\partial\apPtilde_6)$ avoids
$-\Pout[2]$ except for a portion contained in the triangle
$T_3=\triangle(-\lambda x_0,\;-2.5\lambda^2 -.1i,\;-2.5\lambda^2)$.
See \autoref{PicF54}.
The image of $\apPtilde_6 \cap (\R\cup\i\R)$ is contained within the
segment $(-0.30996, -0.30970)$; the remainder of the image exits
through $\mh_+$ before intersecting the real axis at two points, one
inside $T_3$ (more than $3.2\times10^{-4}$ from the vertex of $T_3$)
and the other beyond $+0.0085$; the remaining part
of the image curve gets no closer to $-\Pout[2]$ than this. To see that $T_3$ lies outside $\Piece[2]{0,\III}$,
observe that $F^2(\Piece[2]{0,\III})=\Piece[1]{0,\II}$, but all points
of $F^2(T_3)$ have positive real parts. We obtain that
$\dist(\apF^{54}(\partial\apPtilde_6)\cap\mh_-,\ \Piece[2]{0,\III})>3.2\times10^{-4}$.

\begin{figure}[htb]
 \centerline{
   \includegraphics[height=.27\hsize]{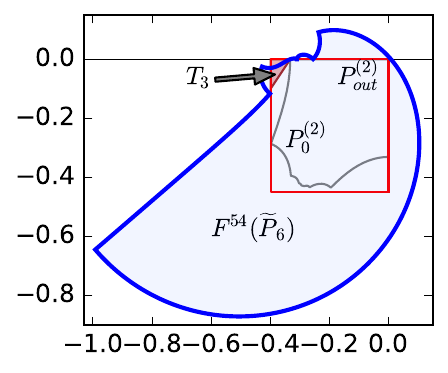}\hfil
   \includegraphics[height=.27\hsize]{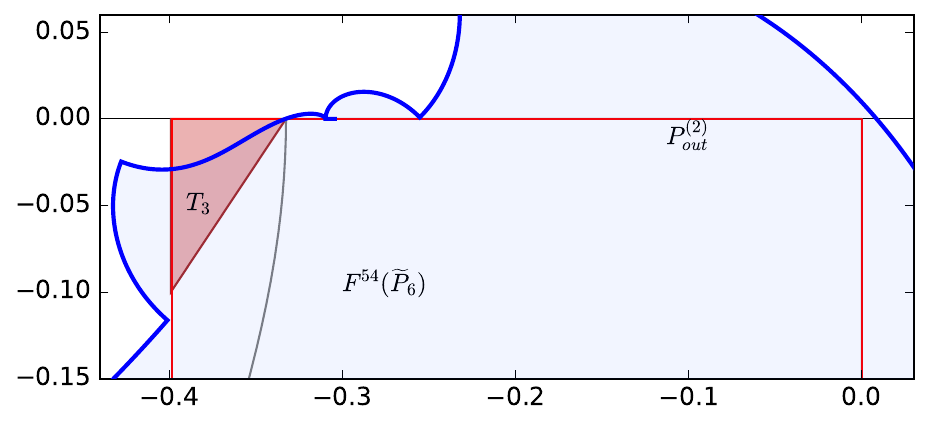}
 }
 \caption{\label{PicF54}
   The set $\apF^{54}(\apPtilde_6)$ is shown shaded in blue, along
   with $\Pout[2]$~(red outline), $\Piece[2]{0,\III}$ (gray outline), and
   $T_3$ (shaded red).  The image on the right is a close-up of the left. }
\end{figure}

To verify that the above calculations are sufficiently accurate we write
$\apF^{54} = \apF^{2}\circ\apF^{4}\circ\apF^{16}\circ\apF^{32}$
and (as noted in \autoref{RemSpeedErrs}) compute $\apF^{2^n}(\apz)$ as
$(-\lambda)^n \apF(\apz/\lambda^n)$; then we have
$|F^{54}(z) - \apF^{54}(\apz)|$ no greater than $1.3\times 10^{-10}$ for all
$z\in\apWtilde_6$, establishing that $\Wtilde[6]\subset\apWtilde_6$
and completing the proof.
\end{proof}

\medskip
\begin{rem}
To approximate $\Wtilde[n]$ for $n>6$,
the set $\Gen[n-6]{\apWtilde_6}$ may be used,
since by
\autoref{WtildeProps}\ref{WtildeScale} and \autoref{LmWapprox} we have
$\Wtilde[n] \subset \lambda^{n-6}\,\apWtilde_6$.
One can use $\apWtilde_4=\CBOX{-.09-.15i}{.09+.15i}$
and $\apWtilde_5=\Gen[5]{(\Pout\cup\CBOX{1.2+1.79i}{3.0+3.65i})}$;
$\Wtilde[9]\subset\Pout[9]$.
However, we do not need to use these.
\end{rem}

Using Lemma \ref{LmWapprox}, we can calculate better approximations of $H$
and $W$ from below and from above. For instance, to find a more accurate
approximation of $\Piece{0,\I}$ from above we use the following.

\begin{algorithm}\label{AlgPoutAcc}
Fix $m\ge 1$ and $t>0$ small.
To compute an upper bound for $\Piece{0,\I}$, we find a collection
$\SS_{m,t}$ of disks with radius $t$ which cover  $\Piece{0,\I}$.
First, select a grid of points $z$ so that
  $\bigcup_z \Disk[t]{z}$ covers $\Pout$. Include in $\SS_{m,t}$ all
the disks $\Disk[t]{z}$ which intersect $\Pin$. For each remaining point $z$ in
the grid, set $w=\lambda^n z$, $q=\lambda^n(0.1+0.1i)$, $t_0=\lambda^m t$ and
run the following routine.
  \begin{enumerate}
    \item \label{PoutIter} For $0< j<2^m$,  calculate $\apw_j=\apF^j(w)$. Using
      the Koebe Theorem and \autoref{PropEvalFError}, find $t_j$ such that
      $\apF^j(\Disk[t_0]{w})\subset \Disk[t_j]{\apw_j}$.
    \item\label{PoutZero}
      If $\Disk[t_j]{\apw_j}$ contains zero, add $\Disk[t]{z}$ to $\SS_{m,t}$ and
      exit the routine.
    \item\label{PoutWrongQuadr}
      If $\Disk[t_j]{\apw_j}$ does not intersect $\Wout$ or the quadrant
      to which $\apF^j(q)$ belongs, discard $z$ and exit the routine.
    \item\label{PoutNothing}
      If for all $0<j<2^m$ neither of the conditions
      \ref{PoutZero} or \ref{PoutWrongQuadr} are satisfied, then add
      $\Disk[t]{z}$ to $\SS_{m,t}$.
  \end{enumerate}
\end{algorithm}

\subsection{Implementing Algorithms~\ref{AlgXn} and \ref{AlgSigma}}

We only need to make a few straightforward substitutions in order to
implement the algorithms for approximating $\tilde{X}_6$ and $\Sigma_6$.

\smallskip
We make the obvious substitution of $F^j$ by $\apF^j$ and $DF^j$ by $D\apF^j$.
When calculating $\apF^j$, we instead compute compositions of $\apF^{2^k}$ as
described in \autoref{RemFSpeedUp}, replacing the set $\W[n]$ by $\Win[n]$.
Whenever practical, we scale $\apz$ so that $z/\lambda^m$ lies in $\Win[1]$,
ensuring that we have $|\apz/\lambda^m|<1.17$. This implies by
\autoref{PropEvalFError} that $\apF$ and $\apF'$
agree with $F$ and $F'$ to within $6.5\times 10^{-13}$ and $5.6\times 10^{-12}$,
respectively.
We take $K=20$.

Further, we replace $\Piece[1]{0,\I}$ by $\Pout[1]$,
$V_2$ by $\Vtwo$, and $\Wtilde[6]$ by $\apWtilde_6$.
When checking whether $\apF^j(\apz) \in \H[1]$ for all $j$
(step~\ref{XnNoescape} of both algorithms, as well as
\autoref{AlgSigma}\ref{SigmaNoEscape}), we use $\Hin[1]$, but otherwise we use
$\Hout[1]$ for $\H[1]$.

\medskip
In calculating the orbit of a point $z$ we keep a running bound on
the accumulated total difference between the true orbit $F^j(z)$ and
the aproximation $\apF^j(\apz)$, as well as the corresponding derivatives.
More specifically, when calculating the iterates
$\apz_{k+1} = \apF^{j_k}(\apz_k)$, we use \autoref{PropEvalFError} and
\autoref{RemCompositionErrs} to compute upper bounds
$\delta_k > |DF^{j_k}(\apz)-D\apF^{j_k}(\apz)|$ and
$\epsilon_k > |F^{j_k}(\apz)-\apF^{j_k}(\apz)|$.

When calculating $r$ in \autoref{AlgXn}\ref{XnCheck},
$\dist(\apF^\ell(\apz),\Vtwo^*)$ should be reduced by $\epsilon_k$ and
$D\apF^\ell(\apz_k)$ should be increased by $\delta_k$; similar changes should
be made in \autoref{AlgSigma}\ref{SigmaCheck}.

\medskip


As long as the points $\apz_k$ remain in $\Win[1]$ for all $k$, we can compute
compositions of $\apF^{2^n}$ with reasonably high precision.\footnote{%
  If $\apz_k\in\Hout[1]\setminus\Win[1]$ for some $k$, we can either use
  $\apz_{k+2}=\apF(\apz_k)=(-1/\lambda)\apF^2(\lambda\apz_k)$ with
  $\lambda\apz_k$ and $F(\lambda\apz_k)$ in $\Win[1]$,
  or use the bounds from \autoref{PropEvalFError} together with
  \autoref{RemCompositionErrs} to calculate $\epsilon$ and $\delta$ for
  the remainder of the orbit.}
In particular, if $\FLT{g}_k$ is a $k$-fold composition of such
approximations and $g_k$ is the same composition of Feigenbaum maps
$F^{2^n}$, we have the following worst-case bounds on $\epsilon_k$
and~$\delta_k$ for $\apz_k\in\Win[1]$.

\begin{center}\begin{tabular}{rc@{\;\;}lc@{\;\;}l}
$k$ &&
{$|g_k(z)-\FLT{g}_k(\apz)|$}    &&
{$|g_k'(z)-\FLT{g}_k'(\apz)|$} \\[2pt]
\hline\\[-2ex]
 1 && $6.45 \times 10^{-13}$ &&         $5.59 \times 10^{-12}$ \\
 5 && $2.15 \times 10^{-10}$ &&         $4.45 \times 10^{-9 }$ \\
10 && $2.14 \times 10^{-7}$  &&         $1.43 \times 10^{-5}$ \\
15 && $2.13 \times 10^{-4}$  &&         $4.57 \times 10^{-2}$ \\
18 && $1.20 \times 10^{-2}$  &&  $\!\!\!15.14$
\end{tabular}
\end{center}

The above bounds are the worst case; actual calculated orbits have
significantly better bounds as long as they remain within $\Win[1]$.
As noted earlier, we compute sharper bounds on the specific function values and
derivatives for each point $\apz_k$, and incorporate these into our
calculations in the implementation of the algorithms.
This ensures that all calculated orbits are shadowed by true orbits under $F$.


\medskip
To estimate an upper bound on $M_n =  M\bigl(A)$ (see \autoref{PropKoebeBA0})
with $A=(\lambda^{-n}\Sigma_{n}) \cap \Piece{0,\I}$,
we replace the integral in the definition of $g_A(z)$ by the Riemann sum taken
over the centers of the subset $\EE_{n,r}$ of $A$ from
\autoref{AlgSigma}, that is
\[ g_A(z)\ge g_{\EE_{n,r}}(z) \approx \sum_{w_k} \frac{2r^2}{C^2(z, w_k)},
   \quad\text{as $w_k$ ranges over centers of the disks in $\EE_{n,r}$}.
\]
We then approximate $M(A)\le M(\EE_{n,r}) \approx \max( 1/g_{\EE_{n,r}}(z_{j}))$, where
$z_{j}$ ranges over centers of the disks from the covering $\SS_{m,t}$ of
$\Piece{0,\I}$ from   \autoref{AlgPoutAcc} (we take $m=2,t=2^{-6}\sqrt 2$).
The relative error in this approximation of $M(A)$ can be bounded by noticing
that
$$C(z,w)\leq C(z,z_j)\,C(z_j,w_k)\, C(w_k,w);$$
for points $w\in \Disk[r]{w_k}$, the Koebe Theorem gives
$$
C(w_k,w)\leq \frac{1-r/R}{(1+r/R)^3}
\quad\text{with $R=\dist(w_k,\,\partial\C_\lambda)$}.
$$
Recall from \autoref{PropC2} that
$\C_\lambda=\C\setminus\left(
  (-\infty,-\tfrac{1}{\lambda}]\cup[\tfrac{F(\lambda)}{\lambda^2},\infty)
 \right)$.

\medskip
In implementing \autoref{AlgXn} to obtain $\DD_{n,r} \supset X_n$ and
\autoref{AlgSigma} for $\EE_{n,r} \subset \Sigma_n$, we take $n=6$ in
both cases.  For $\DD_{n,r}$ we use $r=2^{-17}\sqrt2$; for $\EE_{n,r}$,
taking $r=2^{-11}\sqrt2$ is sufficient.


\ifthenelse{\IsThereSpaceOnPage{.25\textheight}}{\clearpage}{}


\begin{thebibliography}{BPZ07}
\newcommand\DOI[1]{\quad{}\href{http://dx.doi.org/#1}{doi:#1}.}
\newcommand\arXiv[1]{\quad{}\href{http://arxiv.org/abs/#1}{arXiv:#1}.}
\setlength{\itemsep}{5pt plus 0.5ex}

\bibitem[AL08]{AvilaLyubich-08}
A. Avila and M. Lyubich,
\textit{Hausdorff dimension and conformal measures of Feigenbaum Julia sets}.
J. Amer. Math. Soc.~{\bf 21} (2008), pp.~305--383.
\DOI{10.1090/S0894-0347-07-00583-8}

\bibitem[AL15]{AvilaLyubich-15}
A. Avila and M. Lyubich,
\textit{Lebesgue measure of Feigenbaum Julia sets}.
\arXiv{1504.02986}

\bibitem[BPZ07]{BrentPercivalZimmermann}
R. Brent, C. Percival and P. Zimmermann,
\textit{Error Bounds on Complex Floating-Point Multiplication}.
Math. Comp.~{\bf 76}, no. 259 (2007), pp.~1469--1481.
\DOI{10.1090/S0025-5718-07-01931-X}

\bibitem[Bu99]{Buff}
X. Buff,
\textit{Geometry of the Feigenbaum map}.
Conform. Geom. Dyn.~{\bf 3} (1999), pp.~79--101 (electronic).
\DOI{10.1090/S1088-4173-99-00031-4}

\bibitem[CT78]{CT}
  P. Coullet and C. Tresser,
  \textit{It\'eration d'endomorphismes et groupe de renormalisation}.
  J. Phys. Colloque~{\bf 39}, no.~C5 (1978), pp.~25--28.
  \DOI{10.1051/jphyscol:1978513}

\bibitem[DH85]{DH}
  A. Douady and J. Hubbard,
  \textit{On the dynamics of polynomial-like maps}.
  Ann. Sci. \'Ec. Norm. Sup.~{\bf 18} (1985), pp.~287--344.
  \DOI{10.24033/asens.1491}


\bibitem[DY16]{DudkoYampolsky-16}
  A. Dudko and M. Yampolsky,
  \textit{Poly-time Computability of the Feigenbaum Julia set}.
  Ergodic Theory and Dynam. Sys.~{\bf 36} (2016), no. 8, pp.~2441-2462.
  \DOI{10.1017/etds.2015.24}

\bibitem[Du]{Duren}
  P. L. Duren,
  \textit{Univalent Functions}.
  Springer, 1983.

\bibitem[Ep89]{Epstein-89}
H. Epstein,
\textit{Fixed points of composition operators II}.
Nonlinearity~{\bf 2}, no. 2 (1989), pp.~305--310.
\DOI{10.1088/0951-7715/2/2/006}


\bibitem[Eps]{Epstein-Notes}
H. Epstein,
\textit{Fixed points of the period-doubling operator}.
Lecture notes, Lausanne 1992.

\bibitem[Go91]{Goldberg}
D. Goldberg,
\textit{What every computer scientist should know about floating-point
  arithmetic}.
ACM Computing Surveys~{\bf 23}, no.~1 (1991), pp.~5--48.
\DOI{10.1145/103162.103163}

\bibitem[HS14]{FeigPolyTime}
P. Hertling and C. Spandl,
\textit{Computing a solution of Feigenbaum's functional equation in
  polynomial time}.
Logical Methods in Computer Science~{\bf 10}, no.~4 (2014), pp.~1--9.
\DOI{10.2168/LMCS-10(4:7)2014}

\bibitem[HJ93]{HuJiang}
  J.~Hu and Y.~Jiang,
  \textit{The Julia set of Feigenbaum quadratic polynomial is locally connected}.
  Preprint, 1993.

\bibitem[Fe78]{Feigenbaum1}
M. Feigenbaum,
\textit{Quantitative universality for a class of class of non-linear
transformations}.
J. Stat. Phys.~{\bf 19} (1978), pp.~25--52.
\DOI{10.1007/BF01020332}

\bibitem[Fe79]{Feigenbaum2}
M. Feigenbaum,
\textit{The universal metric properties of non-linear transformations}.
J. Stat. Phys.~{\bf 21} (1979), pp.~669--706.
\DOI{10.1007/BF01107909}

\bibitem[Hi02]{Higham}
N. Higham,
\textit{Accuracy and stability of numerical algorithms}, 2nd ed.
SIAM, 2002.

\bibitem[IEEE]{IEEE754}
  \textit{IEEE Standard for Floating-Point Arithmetic}.
  IEEE Standard 754-2008 (2008), pp.~1--58.\\
  \DOI{10.1109/IEEESTD.2008.4610935}

\bibitem[J00]{Jiang}
  Y.~Jiang,
  \textit{Infinitely renormalizable quadratic polynomials}.
  Trans. Amer. Math. Soc {\bf 352}, no.~11 (2000), pp.~5077--5091.
  \DOI{https://doi.org/10.1090/S0002-9947-00-02514-9}

\bibitem[La82]{Lanford}
O. Lanford III.
\textit{A computer-assisted proof of the Feigenbaum conjectures}.
Bull. Amer. Math. Soc. (N.S.) {\bf 6}, no.~3 (1982), pp.~427--434.
\DOI{10.1090/S0273-0979-1982-15008-X}

\bibitem[LS05]{LevinSwiatek1}
  G. Levin and G. \'{S}wi\k{a}tek,
  \textit{Hausdorff dimension of Julia sets of Feigenbaum polynomials with
    high criticality}.
  Comm. Math. Phys. {\bf 258}, no. 1 (2005), pp.~135--148.
  \DOI{10.1007/s00220-005-1332-7}

\bibitem[LS10]{LevinSwiatek2}
  G. Levin and G. \'{S}wi\k{a}tek,
  \textit{Measure of the Julia set of the Feigenbaum map with infinite
    criticality}.
  Ergodic Theory and Dynam. Sys. {\bf 30}, no. 3 (2010), pp.~855--875.
  \DOI{10.1017/S0143385709000340}

\bibitem[Ly97]{L} M. Lyubich,
\textit{Dynamics of quadratic polynomials, I-II}.
Acta Math.~{\bf 178}, no.~2 (1997), pp.~185--297.
\DOI{10.1007/BF02392694}

\bibitem[Ly99]{Lyubich-Hairiness}
M. Lyubich
\textit{Feigenbaum-Coullet-Tresser Universality and Milnor's Hairiness
  Conjecture}.
Annals of Math.~{\bf 149}, no.~2 (1999),  pp.~319--420.
\DOI{10.2307/120968}

\bibitem[McM]{McM}
  C. McMullen,
  \textit{Renormalization and 3-manifolds which fiber over the circle}.
  Annals of Math.\ Studies vol.~{\bf 142},
  Princeton University Press, 1996.

\bibitem[dMvS]{dMvS}
  W. de~Melo and  S. van~Strien.
  \textit{One-dimensional dynamics}.
  Springer, 1993.

\bibitem[Mi]{Milnor}
  J. Milnor,
  \textit{Dynamics in one complex variable: Introductory lectures}, 3rd~ed.
  Annals of Math.\ Studies vol.~{\bf 160},
  Princeton University Press, 2006.


\bibitem[Wi]{Wilkinson}
J. H. Wilkinson,
\textit{Rounding Errors in algebraic processes}.
Prentice-Hall,
1963.

\end{thebibliography}
\end{document}